\documentclass[12pt]{article}

\usepackage{amsmath,amsthm,amsfonts,amssymb, color, dsfont}

\usepackage{tikz-cd}  
\usepackage{comment} 

\newtheorem{theorem}{Theorem}[section]
\newtheorem{corollary}[theorem]{Corollary}

\newtheorem{lemma}[theorem]{Lemma}
\newtheorem{definition}[theorem]{Definition}
\newtheorem{remark}[theorem]{Remark}
\newtheorem{assumption}{Assumption}

\def\cB{\mathcal{B}}

\def\cF{\mathcal{F}}

\def\cP{\mathcal{P}}

\def\bE{\mathbb{E}}
\def\bN{\mathbb{N}}

\def\bR{\mathbb{R}}

\begin{document}

\title{Stochastic wave equation with heavy-tailed noise: Uniqueness of solutions and past light-cone property}

\author{Juan J. Jim\'enez  \footnote{University of Ottawa, Department of Mathematics and Statistics, 150 Louis Pasteur Private, Ottawa, Ontario, K1G 0P8, Canada. E-mail address: jjime088@uottawa.ca.} }

\date{September 3, 2024}
\maketitle

\begin{abstract}
\noindent In this article, we study the stochastic wave equation in spatial dimensions \(d \le 2\) with multiplicative L\'evy noise that can have infinite \(p\)-th moments. Using the past light-cone property of the wave equation, we prove the existence and uniqueness of a solution, considering only the \(p\)-integrability of the L\'evy measure \(\nu\) for the region corresponding to the small jumps of the noise. For \(d=1\), there are no restrictions on \(\nu\). For \(d=2\), we assume that there exists a value \(p \in (0,2)\) for which \(\int_{ \{|z|\le 1 \} } |z|^p \nu (dz) < + \infty\).
\end{abstract}


\vspace{1mm}

\section{Introduction}
Let $(\Omega, \mathcal{F}, ( \mathcal{F}_t )_{t \in \mathbb{R}_+}, \mathbb{P})$ be a stochastic basis with the usual conditions of completeness and right-continuity. We consider the stochastic wave equation in spatial dimensions \(d=1, 2\), given by
\begin{equation}
    \label{wave-levy}
    \begin{cases}
        \dfrac{\partial^2 u }{\partial t^2 } (t,x) = \Delta u (t,x) + \sigma(u(t,x)) \dot{ \Lambda } (t,x), \quad & t>0, \, x \in \mathbb{R}^d, \\
        u(0,x) =u_0(x), \quad \dfrac{\partial u }{\partial t} (0,x)= v_0(x), \quad & x \in \mathbb{R}^d,
    \end{cases}
\end{equation}
where $\sigma$ is a globally Lipschitz function, $u_0$ and $v_0$ are assumed to be non-random measurable functions, and $\Lambda= \{ \Lambda(B); B \in \cB_b( \mathbb{R}_+ \times \mathbb{R}^d) \}$ is a \textit{pure-jump L\'evy space-time white noise} given by
\begin{equation}
    \label{levy-noise1}
    \Lambda(B) = b |B| + \int_{B \times \{ |z| \le 1 \}} z \tilde{J}(dt,dx,dz)
          + \int_{B \times \{ |z| > 1 \}} z J(dt,dx,dz),
\end{equation}
where $\cB_b(\mathbb{R}_+ \times \mathbb{R}^d)$ is the class of Borel sets in $\mathbb{R}_+ \times \mathbb{R}^d$ with finite Lebesgue measure, $b \in \mathbb{R}$, $|B|$ is the Lebesgue measure of $B$ in $\mathbb{R}^{d+1}$, $J$ is a Poisson random measure on $\mathbb{R}_+ \times \mathbb{R}^d \times \mathbb{R}$ with intensity $m(dt,dx,dz) = dt \, dx \, \nu(dz)$, and $\tilde{J}$ is the compensated Poisson random measure of $J$ given by $\tilde{J} = J - m$. Here, $\nu$ is a \textit{L\'evy measure} defined on $\mathbb{R}$, i.e., $\nu$ satisfies
\begin{equation}
    \label{levy-measure0}
    \int_{\mathbb{R}} (|z|^2 \wedge 1) \nu(dz) < +\infty \quad \text{and} \quad \nu(\{0\}) = 0.
\end{equation}
 
 We say that a random field \(\phi = \{\phi(t,x) \, ; \, t \ge 0, \, x \in \mathbb{R}^d \}\) is \textit{predictable} if it is measurable with respect to the \(\sigma\)-field \(\tilde{\mathcal{P}} = \mathcal{P}_0 \times \mathcal{B}(\mathbb{R}^d)\), where \(\mathcal{P}_0\) is the predictable \(\sigma\)-field on \(\Omega \times \mathbb{R}_+\), and \(\mathcal{B}(\mathbb{R}^d)\) is the Borel \(\sigma\)-field on \(\mathbb{R}^d\). We denote by \(\mathcal{P}\) the collection of predictable processes. 

A predictable random field \(u = \{ u(t,x) \, ; \, t \ge 0, x \in \mathbb{R}^d \}\) is considered a \textit{mild solution} of \eqref{wave-levy} if it satisfies the following stochastic integral equation:
\begin{equation}
\label{mild-SPDEs}
u(t,x) = w(t,x) + \int_0^t \int_{\mathbb{R}^d} G_{t-s}(x-y) \sigma(u(s,y)) \Lambda(ds,dy),
\end{equation}
where \(G_t(x)\) is the fundamental solution of the wave operator, defined as:
\begin{equation}
\label{fund-sol}
G_t(x) = \begin{cases}
    \frac{1}{2} \mathds{1}_{ \{ |x| < t \} } \quad &\text{if} \quad d=1, \\
    \frac{1}{2 \pi} \frac{1}{ \sqrt{t^2 - |x|^2}} \mathds{1}_{ \{ |x| < t \} } \quad & \text{if} \quad d=2, \\
\end{cases}
\end{equation}
and \(w\) solves the homogeneous wave equation \(\frac{\partial^2 u}{\partial t^2} - \Delta u = 0\) on \(\mathbb{R}_+ \times \mathbb{R}^d\) with initial conditions matching those of \eqref{wave-levy}:
\begin{equation}
\label{convol-int-cond}
w(t,x) = (G_t \ast v_0 )(x) + \frac{\partial}{\partial t} (G_t \ast u_0)(x).
\end{equation}

We assume the following conditions for \(u_0\) and \(v_0\).
\begin{assumption}
\label{ICH}
\(u_0\) and \(v_0\) are deterministic functions with the following properties.
\begin{itemize}
    \item For \(d=1\), \(u_0\) is locally bounded and continuous, and \(v_0\) is locally bounded and measurable.
    \item For \(d=2\), \(u_0\) is continuously differentiable (\(C^1(\mathbb{R}^2)\)), and \(v_0\) is locally $q_0$-integrable with exponent \(q_0 \in (2, \infty]\), i.e., \(v_0 \in L_{\text{loc}}^{q_0}(\mathbb{R}^2)\).
\end{itemize}
\end{assumption}
Under Assumption \ref{ICH}, we obtain,
\begin{equation}
\label{ICH-eq}    
\sup_{t \in [0,T]} \sup_{|x| \le R} |w(t,x)| < +\infty, \quad \text{ for all } T, R \in \mathbb{R}_+.
\end{equation}
\eqref{ICH-eq} can be proved similarly to Lemma 4.2 in \cite{Dalang-Quer} (see also Theorem 1.2 in \cite{Millet-Sole99}).

\medskip

Regarding the theory of stochastic integration, we use the framework developed in \cite{bit1}, which is based on the concept of the {\em Daniell mean}. The stochastic integral on the right-hand side of \eqref{mild-SPDEs} is defined in Appendix A. This integration theory has been successfully applied to studying SDEs and SPDEs with heavy-tailed noises in \cite{balan49, bit2, chong1, chong2, CDH}.

\paragraph{Notation} Throughout this study, we will use the following notation.
\begin{itemize}
    \item $\mathfrak{B}$ denotes the class of bounded domains in $\mathbb{R}^d$.
    \item For $D \in \mathfrak{B}$, $\overline{D}$ represents the topological closure of $D$ with the usual topology.
    \item $\tilde{\mathcal{P}}_b$ is the collection of all sets $A \in \tilde{\mathcal{P}}$ such that there exists $k \in \mathbb{N}$ with $A \subset \Omega \times [0,k] \times [-k,k]^d$.
    \item $||X||_{p} := \mathbb{E} \left[|X|^p \right]$ for $0 < p < 1$ and $||X||_{p} := (\mathbb{E} \left[ |X|^p \right] )^{\frac{1}{p}}$ if $p \ge 1$.
    \item $||X||_{0} := \mathbb{E}[|X| \wedge 1]$, and $||X||_{\infty} := \inf \{ C \ge 0 : P(|X| \le C) = 1 \}$.
    \item $[\![ R,S ]\!] := \{ (\omega, t) \in \Omega \times \mathbb{R}_+ ; R(\omega) \le t \le S(\omega) \}$ for two $\mathcal{F}_t$-stopping times $R$ and $S$.
    \item $(\!( R,S ]\!] := \{ (\omega, t) \in \Omega \times \mathbb{R}_+ ; R(\omega) < t \le S(\omega) \}$ for two $\mathcal{F}_t$-stopping times $R$ and $S$.
    \item For $p \in (0, \infty]$, $B^p$ is the set of all $\phi \in \mathcal{P}$ such that
    \[
     ||\phi||_{p, T} := \sup_{(t, x) \in [0,T] \times \mathbb{R}^d} ||\phi(t,x)||_{p} < +\infty, 
    \]
    for all $T \in \mathbb{R}_+$.
    \item For $p \in (0, \infty]$, $B_{\text{loc}}^p$ is the set of all $\phi \in \mathcal{P}$ such that
    \[
     ||\phi||_{p, T, R} := \sup_{t \in [0,T]} \sup_{|x| \le R} ||\phi(t,x)||_{p} < +\infty,
    \]
    for all $T, R \in \mathbb{R}_+$.
    \item If $\tau$ is a $\mathcal{F}$-stopping time, we denote $\phi \in B_{\text{loc}}^p (\tau)$ if $\phi \mathds{1}_{[\![ 0, \tau]\!]} \in B_{\text{loc}}^p$, i.e.,
    \[
      \sup_{t \in [0,T]} \sup_{|x| \le R} ||\phi(t,x) \mathds{1}_{[\![ 0, \tau]\!]}(t)||_{p} < +\infty,
    \]
    for all $T, R \in \mathbb{R}_+$.
    \item $B_r(x) := \{ y \in \mathbb{R}^d \, ; \, |x - y| < r \}$ for $x \in \mathbb{R}^d$ and $r>0$.
\end{itemize}

The existence and uniqueness of solutions for the stochastic wave equation \eqref{wave-levy}, where \(\Lambda\) is replaced by a {\em Gaussian noise} \(W\), has been extensively studied since the seminal lecture notes by Walsh \cite{walsh86}; for additional references, see \cite{Dalang99, DalangWave, Millet-Sole99}. On the other hand, when $\nu$ satisfies 
\begin{equation}
\label{finite-variance}
\int_{\{|z| > 1\}} |z|^2 \nu (dz) < + \infty,
\end{equation}
\(\Lambda\) induces a square-integrable martingale with discontinuities, making it suitable for theories typically used for \(L^2\)-random measures, such as Gaussian noises. Consequently, under condition \eqref{finite-variance}, the existence and uniqueness of a solution to \eqref{wave-levy} can be established in a manner similar to the Gaussian case. Equation \eqref{wave-levy} in dimension $d=1$, with condition \eqref{finite-variance}, has been studied in \cite{balan28, balan33, balan38}. However, without condition \eqref{finite-variance}, \(\Lambda\) may have an infinite second moment. A well-known case is the \textit{\(\alpha\)-stable L\'evy white noise}, characterized by the L\'evy measure \( \nu_{\alpha} \) given by
\begin{equation}
\label{alpha-levy-m0}
\nu_{\alpha} (dz) = [ c_{+} \alpha z^{- \alpha -1} \mathds{1}_{(0, \infty)}(z) + c_{-} \alpha (-z)^{- \alpha -1} \mathds{1}_{(-\infty, 0)} (z)] \, dz,
\end{equation}
where \(c_{+}, c_{-} \ge 0\) and \(\alpha \in (0,2)\). Note that \(\int_{\mathbb{R}} |z|^2 \nu_{\alpha} (dz) = + \infty\), implying that \(\mathbb{E}|\Lambda(B)|^2= + \infty \) for all $B \in \mathcal{B}_b (\mathbb{R}_+ \times \mathbb{R}^d)$ with positive Lebesgue measure.

\medskip

Despite the extensive literature on \eqref{wave-levy} driven by \(L^2\)-random measures, to our knowledge \cite{balan49} is the only work that addresses the stochastic wave equation \eqref{wave-levy} driven by a multiplicative L\'evy white noise which may have infinite variance. Specifically, in \cite{balan49}, it was proved the existence of a mild solution to \eqref{wave-levy} if $\nu$ satisfies the following conditions:
\begin{equation}
    \label{cond-Balan}
    \begin{cases}
         \int_{\{|z| > 1\}} |z|^q \nu (dz) < + \infty \quad & \text{if  \(d=1\), for some  \(q \in (0,2),\)} \\[10pt]
         \int_{\{|z| \le 1\}}|z|^p \nu (dz) +  \int_{\{|z| > 1\}} |z|^q \nu (dz) < + \infty \quad & \text{if \(d=2,\) for some \(0 < q \le p < 2\).}\\
    \end{cases}
\end{equation}
The regularity of the solution paths is also studied in \cite{balan49}. The techniques used in \cite{balan49} to construct a solution to \eqref{wave-levy} are based on the breakthrough results in \cite{chong1} related to the existence of a mild solution of the heat equation in $\mathbb{R}^d$.

\medskip

The goal of this article is to establish the existence and uniqueness of solutions for the stochastic wave equation \eqref{wave-levy} under conditions which are weaker than \eqref{cond-Balan}. 

In Section 2, we prove the existence of a unique (up to modifications) mild solution \(u\) of \eqref{wave-levy} that satisfies \(u \in B^p_{\text{loc}} (T_N)\) for all \(N \in \mathbb{N}\), where \(\{T_N\}_{n \ge 1}\) is an increasing sequence of stopping times with \(T_N \to +\infty\) as \(N \to +\infty\). The main novelty of this section is the uniqueness of a solution to \eqref{wave-levy} for the class of random fields that lie in \(B^p_{\text{loc}} (T_N)\) for all \(N \in \mathbb{N}\), employing the same techniques and stopping times used in \cite{balan49,chong1}. Furthermore, we extend these results to a broader class of wave equations.

In Section 3, we use a different strategy to show the existence and uniqueness of solutions to \eqref{wave-levy} in a finite time interval, under conditions which are weaker than \eqref{cond-Balan}. To be precise, by employing {\em the past light-cone property} (PLCP) of the wave equation, in Theorem \ref{wave_D}, we construct a solution to \eqref{wave-levy} without imposing the condition $\int_{\{|z| > 1\}} |z|^q \nu (dz) < + \infty$ for some $q>0$. In particular, our results show that:
\begin{itemize}
    \item If \(d=1\), there exists a unique solution to \eqref{wave-levy} in the interval \([0,T]\), for a fixed \(T>0\), under Assumption \ref{ICH}.
    \item If \(d=2\), there exists a unique solution to \eqref{wave-levy} in the interval \([0,T]\), for a fixed $T>0$, under Assumption \ref{ICH} and \(\int_{\{|z| \le 1\}} |z|^p \nu (dz) < +\infty\) for some \(p \in (0,2)\).
\end{itemize}
Our method for constructing a solution to equation \eqref{wave-levy}, using the PLCP, differs from the method in Section 2. For this, we use similar techniques as in \cite{balan27, CDH} for solving SPDEs on bounded domains. Additionally, we would like to point out that the uniqueness of the solution, using the PLCP approach in Section 3, is obtained in a different class of random fields compared to Section 2. Hence, it is natural to wonder how these two solutions are related. In Theorem \ref{time-comp}, we prove that these two solutions are identical almost surely for all \( (t,x) \in [0,T] \times \mathbb{R}^d \).

A core principle used throughout this article is the fact that the fundamental solution \(G_t\) of the wave operator satisfies the following property: for any given point \((t,x) \in \mathbb{R}_+ \times \mathbb{R}^d\), the function \((s,y) \mapsto G_{t-s}(x-y)\) has support in the conic region
\begin{equation}
\label{compact_supp}
\mathcal{C}_{t,x}:=\{(s,y) \in [0,t] \times \mathbb{R}^d \, ; \, |x-y| \le t - s\}.
\end{equation}
The region \( \mathcal{C}_{t,x} \) is called \textit{the past light-cone} or \textit{the domain of dependence}. In physics, the past light-cone illustrates causality, ensuring that the effects at a point are only due to sources within this cone. This ensures that solutions to the wave equation adhere to the principle of causality, i.e., the information or energy can only travel within the constraint set by the speed of wave propagation (see Theorem 14.1 of \cite{ves}).

The PLCP has also been used in \cite{DalangWave} for the study of the stochastic wave equation in dimension \(d=3\), driven by a colored Gaussian noise. Unlike the Gaussian noise, which typically influences the entire random field uniformly, a L\'evy noise can introduce abrupt changes or jumps. This makes the analysis of dependencies and influences within the past light-cone crucial for understanding how waves propagate in a heavy-tailed random field.

\medskip

We include a few comments about the stochastic heat equation driven by \(\Lambda\),
\begin{equation}
    \label{semilinear-heat}
    \begin{cases}
    \dfrac{\partial u}{\partial t} (t,x) = \dfrac{1}{2} \Delta u (t,x) + \sigma(u(t,x)) \dot{\Lambda} (t,x), \, \, \, & t > 0, \, x \in \mathbb{R}^d, \\
     u(0,x) = u_0(x),  \, \, \, &  x \in \mathbb{R}^d , \\
    \end{cases}
\end{equation}
where \(u_0\) is a deterministic bounded function on \(\mathbb{R}^d\). A \textit{mild solution} of \eqref{semilinear-heat} is a predictable random field \(u\) that satisfies
\[
u(t,x) = w_0 (t,x) + \int_0^t \int_{\mathbb{R}^d} \rho_{t-s} (x-y) \sigma( u(s,y) ) \Lambda (ds,dy),
\]
where \(\rho_t(x)= (2 \pi t)^{- d/2} \exp\left( - \frac{|x|^2}{2t} \right) \mathds{1}_{\{ t > 0 \}} \) and 
\[
w_0 (t,x) = \int_{\mathbb{R}^d} \rho_t (x-y) u_0 (y) dy.
\]
In \cite{SLB}, it was proved that \eqref{semilinear-heat} has a unique solution that satisfies
\[
\sup_{(t,x) \in [0,T] \times \mathbb{R}^d} \mathbb{E} \left[|u(t,x)|^p\right] < +\infty,
\]
if the L\'evy measure \(\nu\) satisfies
\[
\int_{\mathbb{R}} |z|^p \nu (dz) < +\infty,
\]
for some \(p \in [1,2]\), with \(p < 1 + \frac{2}{d}\). 

In \cite{chong1}, it was proved for the first time that the heat equation \eqref{semilinear-heat} driven by L\'evy noise \(\Lambda\) has a mild solution with unbounded \(p\)-th moments. More precisely, the main result of \cite{chong1} shows that if there exist exponents \(p\) and \(q\) satisfying \(0 < q \le p < 1 + \frac{2}{d}\) and \(\frac{p}{1+(1 + \frac{2}{d} - p)} < q\), such that
\[
\int_{\{ |z| \le 1 \}} |z|^p \nu (dz) < +\infty \quad \text{and} \quad \int_{\{|z| > 1\}} |z|^q \nu (dz) < +\infty,
\]
then, the stochastic heat equation,
\begin{equation}
    \label{semilinear-heat-trunc}
    \begin{cases}
    \frac{\partial u}{\partial t} (t,x) = \frac{1}{2} \Delta u (t,x) + \sigma(u(t,x)) \dot{\Lambda}_N (t,x), \, \, \, & t > 0, \, x \in \mathbb{R}^d, \\
     u(0,x) = u_0(x),  \, \, \, &  x \in \mathbb{R}^d , \\
    \end{cases}
\end{equation}
has a mild solution \(u^{(N)}\) in the space \(B^p_{\text{loc}}\) for each \(N \in \mathbb{N}\), where \(\Lambda_N\) is the truncated noise given by
\begin{equation}
    \label{levy-noise0}
    \Lambda_N (B) =  b |B|  +  \int_{B \times \{ |z| \le 1 \}} z  \tilde{J}(dt,dx,dz) + \int_{B \times \{ 1 < |z| \le N h(x) \}} z J(dt,dx,dz),
\end{equation}
for $B \in \mathcal{B}_b ( \bR_+ \times \bR^d)$, and $ h(x) = 1 + |x|^\eta,$ with \(\eta > d/q\). Moreover, the random field $u$ defined by \(u(t,x) := u^{(N)} (t,x)\) on $\{ t \le \tau_N \}$ is a mild solution to \eqref{semilinear-heat}, where \(\tau_N\) is the stopping time given by
\begin{equation}
    \label{stop-time}
    \tau_N := \inf \left\{ T \in \mathbb{R}_+ \, ; \, \int_0^T \int_{\mathbb{R}^d} \int_{\mathbb{R}} \mathds{1}_{ \{|z| > N h(x) \}} J(dt,dx,dz) > 0  \right\},
\end{equation}
for each \(N \in \mathbb{N}\).

To the best of our knowledge, the uniqueness of solutions of equation \eqref{semilinear-heat} for a globally Lipschitz function $\sigma$ remains an open problem, with the exception of the case $\sigma(u) = \beta u$, when $\beta>0$. As mentioned on page 13 in \cite{chong1}, the main issue in finding a unique mild solution to \eqref{semilinear-heat} is that it does not seem possible to find a complete subspace of $B^p_{\text{loc}}$ such that the stochastic-integral operator $\mathcal{J}_{N}$ given by
\begin{equation}
    \label{trunc-operator_heat}
    \mathcal{J}_{N}(\phi)(t,x) := w(t,x) + \int_0^t \int_{\mathbb{R}^d} \rho_{t-s}(x-y) \sigma(\phi(s,y)) \Lambda_N(ds,dy), \quad \text{for $\phi \in \mathcal{P}$},
\end{equation}
is a self-map. Consequently, due to the lack of the self-map property of $\mathcal{J}_{N}(\phi)$, it is not possible to establish the uniqueness of solutions for equation \eqref{semilinear-heat} via the Banach fixed-point theorem. A different strategy was employed in \cite{berger}, where it was demonstrated that there exists a unique mild solution to \eqref{semilinear-heat} when $\nu(-\infty, 0) = 0$ and $\sigma(u) = \beta u$, with $\beta > 0$.

\section{Existence and Uniqueness}
In this section, we establish the existence and uniqueness of a solution to \eqref{wave-levy} using the same approach as in \cite{balan49}. The primary novelty of this section, compared with the results of \cite{balan49}, is the uniqueness of a mild solution to \eqref{wave-levy}. More precisely, using \eqref{compact_supp}, we show that \eqref{wave-levy} has a unique (up to modification) mild solution that lies in \(B_{\text{loc}}^p (\tau_N)\) for each \(N \in \mathbb{N}\), where \(\tau_N\) is given by \eqref{stop-time}. Before presenting the main results of this section, we provide some preliminary results.

For the subsequent lemma, we use the following notation.
\begin{itemize}
    \item $\mathcal{G}: \mathbb{R}_+ \times \mathbb{R}^d \rightarrow \mathbb{R}$ is a measurable non-negative function.
    \item $\textbf{g}_p(t,x) := \mathcal{G}_{t}^p(x) + \mathcal{G}_{t}(x) \mathds{1}_{\{p \ge 1\}}$ for $p > 0$.
    \item $\mathcal{I}^{(t,x)}(\phi)(s,y) := \mathcal{G}_{t-s}(x-y) \phi(s,y) \mathds{1}_{\{t > s\}}$, for any $(t,x) \in \mathbb{R}_+ \times \mathbb{R}^d$ and $\phi \in \mathcal{P}$.
    \item $|| \cdot ||_{\Lambda_N,p}$ denotes the Daniell mean of $\Lambda_N$.
    \item Given \(\phi \in \mathcal{P}\), we define the random field \(\mathfrak{T}_N(\phi)\) given by
    \[
    \mathfrak{T}_N ( \phi ) (t,x) := \int_0^t \int_{\mathbb{R}^d} \mathcal{G}_{t-s}(x-y) \sigma( \phi (s,y) ) \Lambda_N(ds,dy),
    \]
    for all \((t,x) \in \mathbb{R}_+ \times \mathbb{R}^d\).
\end{itemize}

\begin{lemma}[Lemma 3.3 in \cite{chong1}]
\label{lem33chong}
Let $T > 0$ and $N \in \mathbb{N}$ be fixed. Assume that there exist $0 < q \le p$ such that 
\[
\int_0^T \int_{\mathbb{R}^d } \textbf{g}_p (t,x) \, dx \, dt < +\infty, \, \, \text{and} \, \, \int_{\{ |z| \le 1 \} } |z|^p \nu (dz) + \int_{\{ |z| > 1 \} } |z|^q \nu (dz) < +\infty.
\]
Additionally, if $p < 1$, we assume that $b = \int_{\{ | z | \le 1 \} } z \nu (dz)$. Then, we have the following estimations.
\begin{enumerate}
    \item For any $(t,x) \in [0,T] \times \mathbb{R}^d$, there exists a constant $C= C(T,N,p) >0$ such that for all $\phi \in \mathcal{P}$, we have: 
\begin{equation}
    \label{bdgl1}
    \begin{split}
        & \mathbb{E} \left[ | \mathfrak{T}_N ( \phi ) (t,x) |^p \right] \le || \mathcal{I}^{(t,x)}(\sigma(\phi)) ||_{\Lambda_N,p}^{p \vee 1} \\
        & \le C \int_0^t \int_{\mathbb{R}^d} \textbf{g}_p (t-s,x-y) \left( 1 + \mathbb{E} \left[ | \phi(s,y) |^p \right] \right) h(y)^{p-q} \, dy \, ds.
    \end{split}
\end{equation}
    \item For any $(t,x) \in [0,T] \times \mathbb{R}^d$, there exists a constant  $C= C(T,N,p) >0$ such that for all $(t,x) \in [0,T ] \times \mathbb{R}^d$ and $\phi_1, \, \phi_2 \in \mathcal{P}$ with $\mathfrak{T}_N ( \phi_1 ) (t,x), \; \mathfrak{T}_N ( \phi_2 ) (t,x) < + \infty$ a.s., we have:
\begin{equation}
    \label{bdgl2}
    \begin{split}
        & \mathbb{E} \left[ \left| \mathfrak{T}_N ( \phi_1 ) (t,x)  - \mathfrak{T}_N ( \phi_2 ) (t,x) \right|^p \right]  \le || \mathcal{I}^{(t,x)}(\sigma(\phi_1) - \sigma(\phi_2)) ||_{\Lambda_N,p}^{p \vee 1} \\
        & \le C \int_0^t \int_{\mathbb{R}^d} \textbf{g}_p (t-s,x-y) \mathbb{E} \left[ | \phi_1(s,y) - \phi_2(s,y) |^p \right] h(y)^{p-q} \, dy \, ds.
    \end{split}
\end{equation}
\end{enumerate}
\end{lemma}
The case \( p \ge 2 \) in Lemma \ref{lem33chong} follows by using the maximal inequality (15) in \cite{balan33}.

Lemma \ref{lem33chong} will be a fundamental tool throughout this section. For the application of Lemma \ref{lem33chong}, we have to consider the following assumption on \(\nu\). In particular, if $d=1$, we can extend the constraint \( q \in (0,2) \) in \eqref{cond-Balan} to \( q \in (0, +\infty) \).

\begin{assumption}
\label{wave-levy-assumption}
\textit{(i)} For \(d=1\), there exists \(q \in (0, +\infty)\) such that 
    \[
        \int_{\{|z| > 1\}} |z|^q \nu(dz) < + \infty.
    \]
\textit{(ii)} For \(d=2\), there exist \(0 < q \le p < 2\) such that
    \begin{equation}
        \label{condLevymeasure_d=2}
        \int_{\{|z| \le 1\}} |z|^p \nu(dz) + \int_{\{|z| > 1\}} |z|^q \nu(dz) < + \infty.
    \end{equation}
    Additionally, if \(p < 1\), we assume \(b = \int_{\{|z| \le 1\}} z \nu(dz)\).
\end{assumption}

Notice that the fundamental solution $G_t (x)$ of the wave operator given by \eqref{fund-sol} satisfies:
\begin{equation}
\label{fund-G}
 \int_{\mathbb{R}^d} G_{t}^p (x) dx = \begin{cases}
    2^{1-p} \, t \quad  & \text{for any $p>0$ if } d=1, \\
    \frac{(2 \pi)^{1-p}}{2-p} \, t^{2-p} \quad & \text{for any $p \in (0,2)$ if \(d=2\)},
\end{cases}
\end{equation}
for all $t \in \mathbb{R}_+$. We denote $g_p(t,x) = G_{t}^p (x) + G_{t}(x) \mathds{1}_{\{p \ge 1\}}$. In the following remark, we explain why we can extend the value of $q$ to interval $(0, + \infty)$.
\begin{remark}
\label{d1-small}
For equation \eqref{wave-levy} in dimension \(d=1\), there is no need to impose the \(p\)-integrability condition on the small jumps, i.e., \(\int_{\{ |z| \le 1 \} } |z|^p \nu (dz) < + \infty\) since this condition is automatically satisfied for all \(p \ge 2\) by \eqref{levy-measure0}. Note that if \( d=1\), then \(\int_0^T \int_{\mathbb{R}} G_t^p (x) \, dx \, dt < +\infty\) for all \(p > 0\). Moreover,
\[
\int_{\{ |z| \le 1 \} } |z|^p \nu (dz) \le \int_{\{ |z| \le 1 \} } |z|^2 \nu (dz) < +\infty, \quad \text{for all } p \ge 2.
\]
Thus, for \(d=1\), we can apply Lemma \ref{lem33chong} with \(\mathcal{G}_t = G_t\), and choose any value \(p \ge 2 \vee q\). Conversely, for equation \eqref{wave-levy} in dimension \(d=2\), we must impose condition \eqref{condLevymeasure_d=2}, since \(\int_0^T \int_{\mathbb{R}^d} G_t^p(x) \, dx \, dt < +\infty\) only holds for \(p \in (0,2).\)
\end{remark}
To proceed with the proofs of the main results, we need the following lemma; its proof can be found in \cite{berger0}.
\begin{lemma}
\label{lemma35}
For any $\beta_1> -1, \ldots, \beta_n > -1$,
\[
\int_{T_n (t)} \prod_{j = 1}^n (t_{j+1} - t_j)^{\beta_j} d t_1  \ldots d t_n = \frac{ \prod_{j=1}^n \Gamma(\beta_j + 1) }{\Gamma( \sum_{j=1}^n \beta_j +n + 1 )} t^{ \sum_{j=1}^n \beta_j + n } ,
\]
where $ T_n (t ) := \{ (t_1, \ldots, t_n ) \in (0,t)^n \, ; \, t_1 < \ldots< t_n \} $ and $t = t_{n+1}$.
\end{lemma}

We follow a strategy similar to \cite{chong1} to establish the existence of a mild solution to the stochastic heat equation \eqref{semilinear-heat}. More precisely, we first show that the stochastic wave equation driven by \(\Lambda_N\),
\begin{equation}
    \label{wave-trunc}
    \begin{cases}
        \dfrac{\partial^2 u }{\partial t^2 } (t,x) = \Delta u (t,x)  + \sigma(u(t,x)) \dot{\Lambda}_N (t,x),  \quad & t > 0, \, x \in \mathbb{R}^d, \, \\
        u(0,x) = u_0(x), \quad \quad \dfrac{\partial u }{\partial t} (0,x) = v_0(x), \quad &  x \in \mathbb{R}^d, \\
    \end{cases}
\end{equation}
has a unique mild solution \(u^{(N)}\) in \(B^p_{\text{loc}}\) for each \(N \in \mathbb{N}\), i.e., \(u^{(N)}\) is the only (up to modifications) random field in \(B^p_{\text{loc}}\) satisfying
\begin{equation}
    \label{mild-trunc}
     u^{(N)}(t,x) = w(t,x) + \int_0^t \int_{\mathbb{R}^d} G_{t-s}(x-y) \sigma( u^{(N)}(s,y) ) \Lambda_N (ds,dy).
\end{equation}
We define the operator $\mathcal{T}_N : \mathcal{P} \to \mathcal{P}$ by
\begin{equation}
    \label{trunc-operator}
    \mathcal{T}_N ( \phi ) (t,x) := w(t,x) + \int_0^t \int_{\mathbb{R}^d} G_{t-s}(x-y) \sigma( \phi (s,y) ) \Lambda_N(ds,dy),
\end{equation}
for any $\phi \in \mathcal{P}$. By Lemma 6.6 of \cite{chong2}, the random field $\mathcal{T}_N (\phi)$ admits a predictable modification. We will always work with this modification.

Utilizing the compact support property of \(G_t\), we establish the self-mapping property of \(\mathcal{T}_N\) in \(B^p_{\text{loc}}\). Subsequently, the existence of a unique mild solution for \eqref{wave-trunc} is a consequence of the Banach fixed-point theorem applied to the operator \(\mathcal{T}_N\). It is relevant to highlight that this particular approach cannot be extended to address the uniqueness of a solution to equation \eqref{semilinear-heat}. This limitation arises from the fact that the operator \(\mathcal{J}_N\), as delineated in \eqref{trunc-operator_heat}, is well-defined within the spaces \(B^p\) to \(B_{\text{loc}}^p\), but lacks self-mapping attributes in \(B_{\text{loc}}^p\), as mentioned in the previous section.

\begin{theorem}
    \label{wave-jimenez0}
Assume that Assumptions \ref{ICH} and \ref{wave-levy-assumption} are satisfied. Then, for any fixed $N \in \mathbb{N}$, equation \eqref{wave-trunc} has a unique (up to modifications) mild solution $u^{(N)}$ that satisfies  
\[
    \sup_{t \in [0,T] } \sup_{ |x| \le R} \mathbb{E} \left[ |u^{(N)} (t,x) |^p \right]  < +\infty ,
\]
for all $T > 0$ and $R > 0$, where $p$ is any arbitrary value such that $p \ge q \vee 2$ if $d=1$, and $p$ is the exponent from Assumption \ref{wave-levy-assumption} if $d=2$.

\end{theorem}

\begin{proof}
\textit{Step 1 ($\mathcal{T}_N$ is a self-map in $B^p_{\text{loc}}$).} Note that for any $0 \le s < t \le T$ and $|x| \le R$, we have:
\begin{equation}
    \label{supp1}
     \text{supp}(G_{t-s}(x - \cdot)) \subseteq \overline{B_T(x)} \subset \overline{B_{T + R}(0)}.
\end{equation}
Hence, 
\begin{equation}
    \label{hhk1}
    h(y)^{p-q} < C_T (1 + |R|^{\gamma}) \quad \text{for all } y \in \overline{B_{T + R}(0)},
\end{equation}
where $\gamma = \eta (p-q)$. Then, by Lemma \ref{lem33chong}, \eqref{hhk1}, and \eqref{supp1}, for any $\phi \in B_{\text{loc}}^p$, we have: 
\begin{equation}
    \label{self-map-T_K}
    \begin{split}
     & \mathbb{E} \left[ \Bigg|  \int_0^t \int_{\mathbb{R}^d} G_{t-s} (x-y) \sigma(\phi (s,y)) \Lambda_N (ds,dy) \Bigg|^p \right] \\
    & \le  C_T  \int_0^t \int_{\bR^d} g_p(t-s, x-y) \, \left(1 + \mathbb{E} \left[ | \phi (s,y) |^p \right] \right) h(y)^{p-q} dy ds \\
    & \le  C_T (1 + |R|^{\gamma})  \int_0^t \int_{\bR^d} g_p(t-s, x-y) \left(1 + \mathbb{E} \left[ | \phi (s,y) |^p \right] \right) dy ds \\
    & \le C_{p,T,R} \left( 1 +  \sup_{s \in [0,T] } \sup_{ |y| \le R + T}  \mathbb{E} \left[ | \phi (s,y) |^p \right] \right). \\ 
    \end{split}
\end{equation}
Therefore, by \eqref{ICH-eq} and \eqref{self-map-T_K}, we get:
\begin{equation}
\label{selfmap001}
\sup_{t \in [0,T]} \sup_{|x| \le R} \mathbb{E} \left[ | \mathcal{T}_N (\phi) (t,x) |^p \right] < +\infty, \qquad \text{for all } T, R \in \mathbb{R}_+.
\end{equation}

\textit{Step 2  (Convergence of the Picard iterations).} In this step, we consider the Picard iteration $u_n^{(N)}=\{u_n^{(N)}(t,x) \, ; \, t \ge 0, \, x \in \mathbb{R}^d \}$ given by: $u_0^{(N)} (t,x) := \Psi_0 (t,x)$, where $\Psi_0$ is an arbitrary element of $B^p_{\text{loc}}$, and $u_n^{(N)} := \mathcal{T}_N (u^{(N)}_{n-1} )$ for all $n \in \bN$, i.e.,
\begin{equation}
    \label{picard-trunc}
     u_n^{(N)} (t,x) = w(t,x) + \int_0^t \int_{\mathbb{R}^d} G_{t-s} (x-y) \sigma(u_{n-1}^{(N)} (s,y)) \Lambda_N (ds,dy), \, \,  \text{for } n \in \mathbb{N}.
\end{equation}
By \eqref{selfmap001}, it follows that $u_n^{(N)} \in B^p_{\text{loc}}$ for all $n \in \mathbb{N}$ by induction over $n$. Next, we will show that $\{ u_n^{(N)} \}_{n \in \mathbb{N}}$ is a Cauchy sequence in $B^p_{\text{loc}}$. By Lemma \ref{lem33chong}-(ii), we have:
\begin{equation}
    \label{iterative-n-ineq}
    \begin{split}
        & \mathbb{E} \left[ | u_{n}^{(N)} (t,x) - u_{n-1}^{(N)} (t,x) |^p \right]  \\
       & \le C_T \int_0^t \int_{\mathbb{R}^d} g_p(t-s,x-y) \mathbb{E} \left[ | u_{n-1}^{(N)} (s,y) - u_{n-2}^{(N)} (s,y) |^p \right] h(y)^{p-q} dy ds.
    \end{split}
\end{equation}
Iterating \eqref{iterative-n-ineq}, we get:
\begin{equation}
    \label{iterative-1}
    \begin{split}
        & \mathbb{E} \left[ | u_{n}^{(N)} (t,x) - u_{n-1}^{(N)} (t,x) |^p \right]  \\
         \le &   C_T^n \int_{T_n(t) } \int_{(\mathbb{R}^d)^n}
        \prod_{i=1}^n  g_p( t_{i+1} - t_{i}, x_{i+1} - x_{i} ) \\
        & \times \prod_{i=1}^n h(x_i)^{p-q} \mathbb{E} \left[ |u_{1}^{(N)} (t_1,x_1) - \Psi_0 (t_1,x_1)|^p \right]   d \textbf{x} d \textbf{t}, \\
    \end{split}
\end{equation}
where $\textbf{t} = (t_1,\ldots, t_n)$, $\textbf{x} = (x_1,\ldots, x_n)$ and we set $t_{n+1} = t$ and $x_{n+1} = x$. For a fixed $\textbf{t}$, note that the function $\mathbb{G}_{\textbf{t}}: (\mathbb{R}^d)^n \to [0, + \infty)$ given by
\[
\mathbb{G}_{\textbf{t}} ( \textbf{x} ):= \prod_{i=1}^n  g_p( t_{i+1} - t_{i}, x_{i+1} - x_{i} ) \mathds{1}_{T_n(t)} (\textbf{t}),
\]
has support in the set
\[
\{ \textbf{x} \in (\mathbb{R}^d)^n \, ; \,   |x_{i+1} - x_i | \le t_{i+1} - t_i, \, \,  \text{for } i =1, \ldots, n \}.
\]
Hence, if $t \in [0,T]$ and $|x| \le R$, the integral in \eqref{iterative-1} can be restricted to the values $\textbf{x}$ in the bounded set
\begin{equation}
    \label{balls-wave-c0}
    \left\{ \textbf{x} \in (\mathbb{R}^d)^n \, ; \, |x_i| \le R + T \, \,  \text{for } i =1, \ldots, n \right\}, 
\end{equation}
since
\begin{equation}
    \label{balls-wave-c}
    |x-x_i| \le  |x - x_n | + \sum_{k=i}^{n-1} |x_{k+1} - x_{k}| \le  (t - t_n) + \sum_{k =i}^{n-1} (t_{k+1} - t_{k})  = t - t_i \le t < T.
\end{equation}
Then, by \eqref{balls-wave-c0}, it follows that
\begin{equation}
    \label{wave-cc15}
    \begin{split}
        & \mathbb{E} \left[ | u_n^{(N)} (t,x) - u_{n-1}^{(N)} (t,x)  |^p \right] \le C_T^n   \sup_{s \in [0,T] } \sup_{ |y| \le R + T } \left( \mathbb{E} \left[ |u_{1}^{(N)} (t,x) |^p \right] + \mathbb{E} \Big[ | \Psi_0 (t,x)|^p \Big] \right)  \\ 
        & \qquad \qquad \qquad \qquad \qquad \qquad \int_{T_n(t)} \int_{(\mathbb{R}^d)^n} \prod_{i=1}^n g_p (t_{i+1} - t_{i}, x_{i+1} - x_{i} ) \prod_{i=1}^n h(x_i)^{p-q}  d \textbf{x} d \textbf{t}.
    \end{split}
\end{equation}
On the restricted set given by \eqref{balls-wave-c0}, we have:
\begin{equation}
    \label{hh_k}
    \prod_{i=1}^n h(x_i)^{p-q} \le \prod_{i=1}^n \Big[ 1 + (R+T)^\eta \Big]^{p - q} \le C_T^n (1 + |R|^{n \gamma}).
\end{equation}
Hence, by \eqref{wave-cc15} and \eqref{hh_k}, we obtain:
\begin{equation}
    \label{wave-ccc2}
    \begin{split}
         \mathbb{E} \left[ | u_n^{(N)} (t,x) - u_{n-1}^{(N)} (t,x)  |^p \right] & \le   C_{T,R}^n  \int_{T_n(t) } \int_{(\mathbb{R}^d)^n}
        \prod_{i=1}^n  g_p (t_{i+1} - t_{i}, x_{i+1} - x_{i} ) d \textbf{x}  d \textbf{t} \\
        & := C^n_{T,R} A_n^{(p)}(t). \\
    \end{split}
\end{equation}
Note that $A_n^{(p)}(t)$ does not depend on $x$. If $p < 1$, by \eqref{fund-G} and Lemma \ref{lemma35}, we get:
\[
A_n^{(p)}(t) = C_p^n \int_{T_n(t)} \prod_{j=1}^n (t_{j+1} - t_j)^\alpha  d \textbf{t} = C_p^n \frac{t^{(\alpha+1) n}}{\Gamma( (\alpha +1)n +1 )}, 
\]
where $C_p$ is a constant that depends on $p$, and
\[
\alpha = \begin{cases}
    1 \quad   & \text{if \(d=1,\)} \\
    2-p \quad & \text{if \(d=2.\)}
\end{cases}
\]
Hence,
\begin{equation}
    \label{A_1_SUP}
    \text{if $p<1$,} \quad \quad \sup_{t \in [0,T]}  A_n^{(p)} (t) \le \begin{cases}
          C_{p}^n \frac{T^{2 n}}{(  n ! )^{2}} \quad & \text{if \(d= 1,\)}  \\
          C_{p}^n \frac{T^{(3-p) n}}{(  n ! )^{3-p}} \quad & \text{if \(d= 2.\)}
      \end{cases}
\end{equation}

Assume that $p \ge 1$. If $d=1$, it holds that $G_t^p (x) = 2^{1-p} G_t (x)$, so we can proceed in the same way as for $p < 1$, which implies $\sup_{t \in [0,T]}  A_n^{(p)} (t) \le 2^{n(1-p)} C_{p}^n \frac{T^{2n}}{(n!)^2}$. If $d=2$, by \eqref{fund-G}, we have:
\[
\int_{\mathbb{R}^2} g_p(t,x) dx = c_p t^{2-p} +t \le  ( c_p + T^{p-1}) t^{2- p}, \qquad \text{with} \quad c_p = \frac{(2 \pi )^{1-p}}{2-p}.
\]
Hence, $\sup_{t \in [0,T]}  A_n^{(p)} (t) \le (c_p + T^{p-1})^n C_{p}^n \frac{T^{(3-p) n}}{(  n ! )^{3-p}}$. Thus,
\begin{equation}
    \label{A_22_SUP}
    \text{if $p \ge 1$,} \quad \quad \sup_{t \in [0,T]}  A_n^{(p)} (t) \le \begin{cases}
         2^{n(1-p)} C_{p}^n \frac{T^{2n}}{(n!)^2} \quad & \text{if \(d= 1,\)}  \\
         (c_p + T^{p-1})^n C_{p}^n \frac{T^{(3-p) n}}{(  n ! )^{3-p}}  \quad & \text{if \(d= 2.\)}
      \end{cases}
\end{equation}   
Therefore, for both cases \(p < 1\) and \(p \ge 1\), it holds
\begin{equation}
    \label{sup_AA}
    \sum_{n \ge 1} C_{T,R,p}^n \sup_{t \in [0,T]} A_n^{(p)}(t) < +\infty.
\end{equation}
By \eqref{wave-ccc2} and \eqref{sup_AA}, \(\{ u_n^{(N)} \}_{n \in \mathbb{N}}\) is a Cauchy sequence in \(B^p_{\text{loc}}\). Hence, there exists an element \(u^{(N)} \in B^p_{\text{loc}}\) such that \(u_n^{(N)} \xrightarrow[]{B^p_{\text{loc}}} u^{(N)}\) as \(n \to +\infty\).

\textit{Step 3 (Existence of the solution).} In this step, we verify that $u^{(N)}$ satisfies \eqref{mild-trunc}. First, we apply Lemma \ref{lem33chong}-(ii) with $\mathcal{G}_t = G_t$. Then,
\begin{equation}
    \label{daniell11}
    \begin{split}
        &\mathbb{E} \left[ \Big| \int_0^t \int_{\mathbb{R}^d} G_{t-s} (x-y) (\sigma(u_n^{(N)} (s,y)) - \sigma(u^{(N)} (s,y))) \Lambda_N (ds,dy) \Big|^p \right] \\
        & \le \left\| \mathcal{I}^{(t,x)} (\sigma(u_n^{(N)} )- \sigma(u^{(N)} )) \right\|_{\Lambda_N,p}^{p \vee 1} \\
        & \le  C_T \int_0^t \int_{\mathbb{R}^d } g_p(t-s,x-y) \mathbb{E} \left[ | u_n^{(N)} (s,y) - u^{(N)}(s,y) |^p \right] h(y)^{p-q} dy ds \\
        & \le C_{T,R,p} \sup_{s \in [0,T]} \sup_{ |y| \le T + R} \mathbb{E} \left[ | u_n^{(N)} (s,y) - u^{(N)} (s,y) |^p \right] \int_0^t \int_{\mathbb{R}^d} g_p (t-s, x-y) dy ds,
    \end{split}
\end{equation}
we used \eqref{supp1} and \eqref{hhk1} in the previous inequality. Now, if we let \(n\) approach infinity in \eqref{daniell11}, we find that for a fixed pair \((t,x) \in \mathbb{R}_+ \times \mathbb{R}^d\), the expression \(\mathcal{I}^{(t,x)}(\sigma(u_n^{(N)}))\) converges to \(\mathcal{I}^{(t,x)}(\sigma(u^{(N)}))\) with the semi-norm \(|| \cdot ||_{\Lambda_N,p}^{p \vee 1}\). This convergence implies that
\[
\lim_{n \to +\infty} \mathcal{T}_N(u_n^{(N)})(t,x) = \mathcal{T}_N(u^{(N)})(t,x) \quad \text{in } L^p(\Omega).
\]
Moreover, we have $u_n^{(N)} = \mathcal{T}_N(u_{n-1}^{(N)})$ for all $n \in \mathbb{N}$, and the sequence $\{ u_n^{(N)} \}_{n \in \mathbb{N}}$ converges to $u^{(N)}$ in the space $B^p_{\text{loc}}$ as $n \to +\infty$. Therefore, we conclude that $u^{(N)}$ satisfies \eqref{mild-trunc}.

\textit{Step 4 (Uniqueness of the solution).} Assume that there exists another process $v^{(N)} \in B^p_{\text{loc}}$ that satisfies \eqref{mild-trunc}, i.e., $\mathcal{T}_N (v^{(N)}) = v^{(N)}$. Then, by Lemma \ref{lem33chong}-(ii), we have:
\[
\begin{split}
    & \mathbb{E} \left[ | u_n^{(N)} (t,x) - v^{(N)} (t,x) |^p \right] \\
    & \le C_T \int_0^t \int_{\mathbb{R}^d} g_p (t-s,x-y) \mathbb{E} \left[ | u_{n-1}^{(N)} (s,y) - v^{(N)} (s,y) |^p \right] h(y)^{p-q} ds dy. \\
\end{split}
\]
Iterating the inequality above as in \eqref{iterative-1}, and following the same steps as in \eqref{wave-ccc2}, we get:
\[
\begin{split}
    & \mathbb{E} \left[ | u_n^{(N)} (t,x) - v^{(N)} (t,x) |^p \right] \\
    & \le  C_{p,T,R}^n \int_{T_n(t)} \int_{(\mathbb{R}^d)^n}
        \prod_{i=1}^n  g_p( t_{i+1} - t_{i}, x_{i+1} - x_{i} ) \mathbb{E} \left[ |u_{1}^{(N)} (t_1,x_1) - v^{(N)}(t_1,x_1)|^p \right]  d \textbf{x} d \textbf{t} \\
    &  \le  \sup_{s \in [0,T]} \sup_{|y| \le R + T} \mathbb{E} \left[ |u_{1}^{(N)} (s,y) - v^{(N)}(s,y)|^p \right] \, C_{p,T,R}^n  A_n^p (t) \to 0,
\end{split}
\]
as $n \to +\infty$. Therefore, \(u_n^{(N)} \to v^{(N)}\) in \(B_{\text{loc}}^p\) as \(n \to +\infty\). Alternatively, \(u_n^{(N)} \to u^{(N)}\) in \(B_{\text{loc}}^p\) as \(n \to +\infty\), which implies \(u^{(N)} = v^{(N)}\) in \(B_{\text{loc}}^p\).
\end{proof}

\begin{theorem}
\label{wave-jimenez1}
Let $\tau_N$ be the stopping time given by \eqref{stop-time}. Under the same assumptions as in Theorem \ref{wave-jimenez0} with $\eta > d/q$, equation \eqref{wave-levy} has a unique (up to modifications) mild solution $u$ that satisfies 
\[
\sup_{t \in [0,T] } \sup_{ |x| \le R} \mathbb{E} \left[ | u(t,x) |^p \mathds{1}_{ {[\![ 0, \tau_N]\!]} } (t) \right]  < +\infty, 
\]
for all $T > 0$ and $R > 0$, where $p$ is any arbitrary value such that $p \ge q \vee 2$ if $d=1$, and $p$ is the exponent in Assumption \ref{wave-levy-assumption} if $d=2$.
\end{theorem}

\begin{proof}
\textit{Step 1 (Existence of the solution).} The existence of a solution to \eqref{wave-levy} follows as in the proofs of Theorem 3.1 in \cite{chong1} and Theorem 3.5 in \cite{chong2} for the heat equation \eqref{semilinear-heat}. First, note that we have:
\begin{equation}
    \label{chomp1}
    \Lambda([0,t] \times A) = \Lambda_N([0,t] \times A)  \quad \text{on} \quad \{ t \le \tau_N \},
\end{equation}
for any $A \in \mathcal{B}_b (\mathbb{R}^d)$. Also note that \eqref{chomp1} holds for the L\'evy basis extension of $\Lambda$ given by Remark \ref{ext_levy}, for which we have: $\Lambda(B \cap ( {[\![ 0, \tau_N]\!]} \times \mathbb{R}^d )) = \Lambda_N(B \cap ( {[\![ 0, \tau_N]\!]} \times \mathbb{R}^d ))$ for all $B \in \tilde{\mathcal{P}}_b$. Therefore, by Lemma \ref{local-int} and \eqref{chomp1}, the random field \(u\) given by
\[
u(t,x) = u^{(1)}(t,x) \mathds{1}_{[\![ 0, \tau_1 ]\!]}(t) + \sum_{N=2}^\infty u^{(N)}(t,x) \mathds{1}_{(\!( \tau_{N-1}, \tau_N ]\!]}(t),
\]
is a mild solution to \eqref{wave-levy}, where \(u^{(N)}\) is the solution to \eqref{wave-trunc} given by Theorem \ref{wave-jimenez0}.

\textit{Step 2 (Uniqueness of the solution).} Assume that $v$ is another solution of \eqref{wave-levy} such that $v \in B_{\text{loc}}^p (\tau_N)$ for all $N \in \mathbb{N}$. By Lemma \ref{local-int}, Lemma \ref{lem33chong}, and \eqref{chomp1}, we have:
\[
\begin{split}
& \mathbb{E} \left[ | (u_n^{(N)} (t,x) - v(t,x)) \mathds{1}_{ {[\![ 0, \tau_N]\!]} } (t) |^p \right] \\
& \le C_{T} \int_0^t \int_{\mathbb{R}^d} g_p (t-s, x-y) \mathbb{E} \left[ | (u_{n-1}^{(N)} (s,y) - v(s,y)) \mathds{1}_{ {[\![ 0, \tau_N]\!]} } (s) |^p \right] h(y)^{p-q} dy ds.
\end{split}
\]
Iterating the inequality above and using the same steps as in \eqref{wave-ccc2}, for $t\in[0,T]$ and $|x| \le R$, we obtain that:
\[
\begin{split}
& \mathbb{E} | (u_n^{(N)} (t,x) - v(t,x)) \mathds{1}_{ {[\![ 0, \tau_N]\!]} } (t) |^p \\
& \le C_{T,R}^n \Big( \sup_{s \in [0,T] } \sup_{ |y| \le R + T } \left[ \mathbb{E} \left[ |u_{1}^{(N)} (s,y) \mathds{1}_{ {[\![ 0, \tau_N]\!]} } (s) |^p \right]  + \mathbb{E} \Big[ | v (s,y) \mathds{1}_{ {[\![ 0, \tau_N]\!]} } (s) |^p \Big]  \right] \Big) A_n^{(p)} (t),
\end{split}
\]
where $A_n^{(p)}(t)$ is given by \eqref{wave-ccc2}. Using the fact that $\sup_{t \in [0,T]} A_n^{(p)} (t) \to 0$ as $n \to +\infty$, we conclude that 
\[
\sup_{t \in [0,T]} \sup_{|x| \le R} \mathbb{E} \left[ | (u_n^{(N)} (t,x) - v(t,x)) \mathds{1}_{ {[\![ 0, \tau_N]\!]} } (t) |^p \right] \to 0 \qquad \text{as } n \to +\infty,
\]
for all $R, T > 0$. On the other hand, note that $u_n^{(N)} \to u$ in $B^p_{\text{loc}} (\tau_N)$ as $n \to +\infty$ for all $N \in \mathbb{N}$. Then, $u(t,x) \mathds{1}_{ {[\![ 0, \tau_N]\!]} } (t) = v(t,x) \mathds{1}_{ {[\![ 0, \tau_N]\!]} } (t)$ a.s., and letting $N \to +\infty$, we obtain that $u(t,x) = v(t,x)$ a.s. for all \( (t,x) \in \mathbb{R}_+ \times \mathbb{R}^d\) due to $\tau_N \uparrow +\infty$ a.s. for $N \to +\infty$.
\end{proof}

Now, we investigate the stochastic wave equation driven by a more general heavy-tailed noise. Consider \( L= \{ L(B); B \in \cB_b( \bR_+ \times \bR^d) \} \) given by
\begin{equation}
    \label{general_levy}
         L (B) =   b |B|  + a W(B) +  \int_{B \times \{ |z| \le 1 \} } z  \tilde{J}(dt,dx,dz)
          +   \int_{B \times \{ |z| > 1 \} } z  J(dt,dx,dz), 
\end{equation}
where \( a>0 \) and \( W \) is a space-time Gaussian white noise, i.e., \(W:= \{W(A) \, ; \, A \in \cB_b (\bR_+ \times \bR^d ) \} \) is a zero mean Gaussian process with covariance \( E[W(A) W(B) ]= |A \cap B | \). In the case \( d=2 \), since \( \int_0^t \int_{\bR^d} G_{t-s}^2 (x-y) ds dy = +\infty \), there is no mild solution of \eqref{wave-levy} driven by \( L \) instead of \( \Lambda \). Therefore, using the same steps of Theorem \ref{wave-jimenez0} and Theorem \ref{wave-jimenez1}, we have the following result. 

\begin{corollary}
\label{coro-d1}
Under Assumptions \ref{ICH} and \ref{wave-levy-assumption}, the stochastic wave equation  
\begin{equation}
    \label{wave-levy-general}
    \begin{cases}
        \dfrac{\partial^2 u }{\partial t^2 } (t,x) = \dfrac{\partial^2 u }{\partial x^2 } (t,x)  + \sigma(u(t,x)) \dot{ L } (t,x),  \quad &t>0, \; x \in \mathbb{R},  \\
         u(0,x) =u_0(x), \quad \quad \dfrac{\partial u }{\partial t} (0,x)= v_0(x),  \quad & x \in \mathbb{R}, \\
    \end{cases}
\end{equation}
has a unique (up to modifications) mild solution \( u \) that satisfies
\[
\sup_{t \in [0,T] } \sup_{ |x| \le R} \bE \left[ | u (t,x) |^p \mathds{1}_{ {[\![ 0, \tau_N]\!]} } (t) \right]  < + \infty, \quad \text{for all $p \ge 2$,}
\]
for all \( T,R \in \bR_+ \) and \( N \in \bN \).
\end{corollary}
To prove Corollary \ref{coro-d1}, we applied the version of Lemma 3.3 from \cite{chong1} for L\'evy noises with a Gaussian component.

\medskip

We can extend Theorem \ref{wave-jimenez1} to a more general class of stochastic wave equations driven by multiplicative noises with a non-linear term \( \sigma(t,x,u) \dot{\Lambda} \) and drift \( f(t,x,u) \), i.e., we consider the stochastic wave equation with $d \le 2$,
\begin{equation}
    \label{wave-quasi}
    \begin{cases}
        \dfrac{\partial^2 u }{\partial t^2 } (t,x) = \Delta u (t,x) + f(t,x,u(t,x)) + \sigma(t,x,u(t,x)) \dot{\Lambda} (t,x), & t > 0, \ x \in \mathbb{R}^d, \\
        u(0,x) = u_0(x), \quad \dfrac{\partial u }{\partial t} (0,x) = v_0(x), & x \in \mathbb{R}^d,
    \end{cases}
\end{equation}
where \( u_0 \) and \( v_0 \) are the same initial conditions as in \eqref{wave-levy}. We impose the following conditions on the processes \( \sigma \) and \( f \).

\begin{assumption}
\label{sigma-assumption}
\( \sigma \) and \( f \) are functions defined as \( \Omega \times \mathbb{R}_+ \times \mathbb{R}^{d+1} \to \mathbb{R} \) which are measurable with respect to \( \tilde{\mathcal{P}} \times \mathcal{B} ( \mathbb{R}) \). In addition, we assume there exist positive processes \( \mathcal{C}_f, \, \mathcal{C}_{\sigma} \in B_{loc}^{\infty} \) such that for all \( (t,x) \in \mathbb{R}_+ \times \mathbb{R}^d \) and \( l_1, \, l_2 \in \mathbb{R} \), we have:
\begin{equation}
    \label{sigma-cond1}
    | \sigma(t,x,l_1 ) - \sigma(t,x,l_2) | \le  \mathcal{C}_{\sigma} (t,x)  \, | l_1 - l_2 | \quad \text{a.s.},
\end{equation}
and
\begin{equation}
    \label{b-cond1}
    | f(t,x,l_1 ) - f(t,x,l_2) | \le  \mathcal{C}_f (t,x)  \, | l_1 - l_2 | \quad \text{a.s.}
 \end{equation}
\end{assumption} 

Denote \( \sigma_0 (t,x) = \sigma(t,x,0) \) and \( f_0 (t,x) = f(t,x,0) \). A \textit{mild solution} to \eqref{wave-quasi} is a predictable random field \(u\) that satisfies
\begin{equation}
    \label{eqvolterra1}
    \begin{split}
        u(t,x) = w(t,x) & + \int_0^t \int_{\mathbb{R}^d} G_{t-s} (x-y) f(s,y, u(s,y)) \, dy \, ds \\
        & + \int_0^t \int_{\mathbb{R}^d} G_{t-s} (x-y) \sigma(s,y, u(s,y)) \, \Lambda(ds,dy). \\
    \end{split}
\end{equation}

\begin{theorem}
\label{jimenez-quasi-u}
Under Assumptions \ref{ICH}, \ref{wave-levy-assumption}, and \ref{sigma-assumption}, the following results hold.
\begin{itemize}
    \item[1.] For \( d=1 \), assume that there exists \( p \ge 2 \vee q \) such that \(\sigma_0\) and \(f_0\) belong to \( B^p_{\text{loc}} \), then \eqref{wave-quasi} admits a unique (up to modifications) mild solution \( u \) satisfying 
    \[
    \sup_{t \in [0,T]} \sup_{|x| \le R} \mathbb{E} \left[ \left| u(t,x) \right|^p \mathds{1}_{[\![ 0, \tau_N ]\!]}(t) \right] < +\infty,
    \]
    for all \( T, R > 0 \) and \( N \in \mathbb{N} \).

    \item[2.] For \( d=2 \), assume that \(\sigma_0\) and \(f_0\) belong to \( B^p_{\text{loc}} \), where \( p \) is the exponent in \eqref{condLevymeasure_d=2}. Additionally, if \( p < 1 \), we impose \( f(t,x,l) = 0 \) a.s. for all \((t,x,l) \in \mathbb{R}_+ \times \mathbb{R}^d \times \mathbb{R}\). Then, equation \eqref{wave-quasi} admits a unique (up to modifications) mild solution \( u \) satisfying
    \[
    \sup_{t \in [0,T]} \sup_{|x| \le R} \mathbb{E} \left[ \left| u(t,x) \right|^p \mathds{1}_{[\![ 0, \tau_N ]\!]}(t) \right] < +\infty,
    \]
    for all \( T, R > 0 \) and \( N \in \mathbb{N} \).
\end{itemize}
 
\end{theorem}

\begin{proof}
 First, we prove that the operator $\Im_N: \cP \to \cP$ given by
\[
\begin{split}
    \Im_N( \phi ) (t,x) :=  w(t,x) & +  \int_0^t \int_{ \mathbb{R}^d } G_{t-s} (x-y) f(s,y, \phi (s,y)) dy ds \\
     &+\int_0^t \int_{ \mathbb{R}^d } G_{t-s} (x-y) \sigma(s,y, \phi (s,y)) \Lambda_N (ds,dy),
\end{split}
\] 
is a self-map in $B_{\text{loc}}^p$, for each fixed $N \in \mathbb{N}$. Note that $\Im_N$ is well-defined since $\sigma(t,x,\phi(t,x))$ and $f(t,x,\phi(t,x))$ are predictable for all $\phi \in \cP$ as a consequence of $\sigma$ and $f$ being measurable with respect to $\tilde{\mathcal{P}} \otimes \mathcal{B}(\mathbb{R})$. Hence, by Lemma 6.2 in \cite{chong2}, $\Im_N( \phi ) (t,x)$ has a predictable modification. Now, by Lemma \ref{lem33chong}, \eqref{supp1}, \eqref{sigma-cond1}, and Hölder's inequality $||X Y ||_1 \le ||X ||_{\infty} || Y ||_1$, for any $t \in [0,T]$ and $|x| \le R$, we have:
\begin{equation}
\label{eq111}
    \begin{split}
        & \mathbb{E} \left[ \Big| \int_0^t \int_{ \mathbb{R}^d } G_{t-s} (x-y) \sigma(s,y, \phi(s,y) ) \Lambda_N (ds,dy) \Big|^p \right] \\
         & \le C_{p,T} \int_0^t \int_{\mathbb{R}^d} g_p(t-s,x-y) \mathbb{E} \left[ | \sigma(s,y, \phi(s,y)) |^p \right] \, h(y)^{p-q} dy ds \\
          \le & C_{p,T,R} \Big( \sup_{s \in [0,T]} \sup_{|y| \le T+R } || \mathcal{C}_{\sigma}(s,y) ||_{\infty}^p \vee \mathbb{E} \left[ | \sigma_0(s,y)|^p \right] \Big) \\
         & \times \Big( \sup_{s \in [0,T]} \sup_{|y| \le T+R } \mathbb{E} \left[ | \phi(s,y) |^p \right] + 1 \Big). \\
    \end{split}
\end{equation}
Hence,
\begin{equation}
    \label{eq112}
    \sup_{t \in [0,T] } \sup_{ |x| \le R } \bE \left[ \Big| \int_0^t \int_{ \mathbb{R}^d } G_{t-s} (x-y) \sigma(s,y, \phi(s,y) ) \Lambda_N (ds,dy) \Big|^p \right] < +\infty .
\end{equation}
Next, we examine the integral that corresponds to the drift \( f \). Recall that if \( p \in (0,1) \), we assume that \( f = 0 \). Hence, we consider only the case \( p \ge 1 \). By Hölder's inequality and \eqref{b-cond1}, for any \( t \in [0,T] \) and \( |x| \le R \),
\begin{equation}
\label{eq113}
\begin{split}
& \mathbb{E} \left[ \Big| \int_0^t \int_{ \mathbb{R}^d } G_{t-s} (x-y) f(s,y, \phi(s,y) ) dy ds \Big|^p \right] \\
& \le \Big( \int_0^t \int_{\mathbb{R}^d} G_{t-s} (x-y) dy ds \Big)^{p-1} \int_0^t \int_{ \mathbb{R}^d } G_{t-s} (x-y) \mathbb{E} \left[ | f(s,y, \phi(s,y)) |^p \right] dy ds \\
& \le C_T \Big( \sup_{s \in [0,T]} \sup_{|y| \le T+R } || \mathcal{C}_{f}(s,y) ||_{\infty}^p \vee \mathbb{E} \left[ | f_0 (s,y) |^p \right] \Big) \Big( \sup_{s \in [0,T]} \sup_{|y| \le T+R } \mathbb{E} \left[ | \phi(s,y) |^p \right] + 1 \Big). \\
\end{split}
\end{equation}
Then,
\[
\sup_{t \in [0,T]} \sup_{|x| \le R } \mathbb{E} \left[ \Big| \int_0^t \int_{ \mathbb{R}^d } G_{t-s} (x-y) f(s,y, \phi(s,y) ) dy ds \Big|^p \right] < +\infty.
\]
Therefore, by \eqref{eq112} and \eqref{eq113},
\begin{equation}
\label{self-maps-fact}
\sup_{t \in [0,T]} \sup_{|x| \le R } \mathbb{E} \left[ | \Im_N (\phi) (t,x) |^p \right]  <+ \infty.
\end{equation}

Now, consider the Picard's iteration sequence \(\{u_n^{(N)} \}_{n \ge 0 }\) given by \(u_0^{(N)} (t,x) := \Psi_0(t,x)\) where \(\Psi_0 \in B^p_{\text{loc}}\), and
\begin{equation}
\label{picard-trunc-quasi}
\begin{split}
u_{n+1}^{(N)}(t,x) := w(t,x) & +  \int_0^t \int_{ \mathbb{R}^d } G_{t-s} (x-y) f(s,y, u_n^{(N)}(s,y)) dy ds \\
&+\int_0^t \int_{ \mathbb{R}^d } G_{t-s} (x-y) \sigma(s,y, u_n^{(N)}(s,y)) \Lambda_N (ds,dy). \\
\end{split}
\end{equation}
By \eqref{self-maps-fact}, it follows that
\[
\sup_{t \in [0,T] } \sup_{ |x| \le R } \mathbb{E} \left[ | u_{n}^{(N)}(t,x) |^p \right] < +\infty,
\]
for all \(R, T > 0\) and \(n \in \mathbb{N}\). Similarly to \eqref{eq112} and \eqref{eq113}, we obtain
\begin{equation}
\label{iterat-ineq}
\begin{split}
    & \mathbb{E} \left[ | u_n^{(N)} (t,x) - u_{n-1}^{(N)} (t,x)  |^p \right]  \\
   & \le C_{T}  \int_0^t \int_{ \mathbb{R}^d }g_p(t-s,x-y) \mathbb{E} \left[|  u_{n-1}^{(N)} (s,y) - u_{n-2}^{(N)} (s,y)   |^p\right]  h(y)^{p-q} ds dy. \\
\end{split}
\end{equation}
Following the same procedure as in the proof of Theorem \ref{wave-jimenez0}, by iterating inequality \eqref{iterat-ineq}, we obtain that:
\[
    \mathbb{E} \left[ | u_n^{(N)} (t,x) - u_{n-1}^{(N)} (t,x)  |^p \right] \le  C_{p,T,R} A_n^{(p)} (t), 
\]
where \(A_n^{(p)} (t)\) is defined as in \eqref{wave-ccc2}. Hence, \(\{ u_n^{(N)} \}_{n \in \mathbb{N} }\) is a Cauchy sequence in \(B_{\text{loc}}^p\). Consequently, there exists a limit \(u^{(N)}\) in \(B_{\text{loc}}^p\). The existence and uniqueness of a solution \(u\) to \eqref{wave-quasi} follow in the same manner as in the proof of Theorem \ref{wave-jimenez1}.

\end{proof}

\section{Past light-cone property}
In this section, we fix \( T > 0 \) and consider the stochastic wave equation in the interval \([0, T]\),
\begin{equation}
    \label{wave-levy-T}
    \begin{cases}
        \dfrac{\partial^2 u}{\partial t^2} (t, x) = \Delta u (t, x) + \sigma(u(t, x)) \dot{\Lambda}(t, x), \quad & t \in (0, T], \, x \in \mathbb{R}^d, \, d \le 2, \\
        u(0, x) = u_0(x), \quad \dfrac{\partial u}{\partial t}(0, x) = v_0(x), \quad & x \in \mathbb{R}^d,
    \end{cases}
\end{equation}
where \( \sigma \) is a globally Lipschitz function, \( u_0 \) and \( v_0 \) satisfy Assumption \ref{ICH}. By exploiting the PLCP of the wave equation, we will show the existence and uniqueness of a solution to \eqref{wave-levy-T} without imposing the assumption of $q$-integrability over the large jumps:
\begin{equation}
\label{large-jumps-sum}
\int_{\{|z| > 1\}} |z|^q \nu (dz) < +\infty.
\end{equation}

The PLCP has been extensively studied for hyperbolic PDEs in the literature. For instance, in \cite{DalangWave}, it was proved that the solution of the stochastic wave equation, in dimension \(d =3\), driven by a colored Gaussian noise remains invariant in a region of space, if the problem is restricted to that region. As in Section 6 of \cite{DalangWave}, for a fixed region $D \in \mathfrak{B}$, we define the conic region
\[
\mathcal{K}^D (s) := \{ y \in \mathbb{R}^d \; ; \; d(y, D) \le T-s \}, \quad \text{for any } s \in [0, T].
\]
Clearly, \(\mathcal{K}^D (0) = \bigcup_{t \in [0, T]} \mathcal{K}^D (t)\).

As we do not impose condition \eqref{large-jumps-sum}, the stopping time \(\tau_N\) given by \eqref{stop-time} may not be well-defined. Therefore, we consider another stopping time given by
\begin{equation}
    \label{stop-time-D}
    \tau_N (D) := \inf \left\{ t \in [0, T] \, ; \, J([0, t] \times \mathcal{K}^D (0) \times \{ |z| > N \}) > 0 \right\}, \quad N \in \mathbb{N}, \; D \in \mathfrak{B}.
\end{equation}
Note that \(\tau_N(D) > 0\) a.s., and \(\tau_N(D) < \tau_{N+1}(D)\) a.s. for all \(N \in \mathbb{N}\). Moreover, \(\tau_N(D) = +\infty\), for large \(N\).

Additionally, for a fixed \((t, x) \in [0, T] \times \overline{D}\), note that \(\text{supp}(H^{(t,x)}) \subset [0, t] \times \mathcal{K}^D (0)\), where
\[
H^{(t,x)}(s,y) = G_{t-s}(x-y) \mathds{1}_{\{t \ge s\}},
\] 
This implies that for a fixed $(t,x) \in [0,T] \times \overline{D}$, the value $u(t,x)$ of the solution of \eqref{mild-SPDEs} only depends on the values of $\Lambda$ on \([0, T] \times \mathcal{K}^D (0)\). Thus, for any $(t,x) \in [0,T] \times \overline{D}$, if $\phi$ is an integrable random field with respect to $\Lambda$ (see \eqref{def-int-levy}),
\begin{equation}
\label{local_D}
\begin{split}
     & \mathds{1}_{[\![  0, \tau_N(D) ]\!]} (t) \int_0^t \int_{\bR^d} G_{t-s} (x-y) \sigma (\phi (s,y) ) \Lambda(ds,dy)\\
     & =   \mathds{1}_{[\![  0, \tau_N(D)]\!]} (t) \int_0^t \int_{\bR^d} G_{t-s} (x-y) \sigma (\phi (s,y) ) \mathds{1}_{[\![  0, \tau_N(D)]\!]} (s)  \overline{\Lambda}_N(ds,dy), \\  
\end{split}
\end{equation}
due to Lemma \ref{local-int}, and the fact that
\[
\Lambda([0,t] \times A) = \overline{\Lambda}_N ([0,t] \times A) \quad \text{on} \quad \{ t \le \tau_N(D)  \},
\] 
for all $A \in \mathcal{B}_b ( \bR^d )$, with $A \subset \mathcal{K}^D (0).$

\medskip

The stopping times defined in \eqref{stop-time-D} are analogous to those used in \cite{balan27, CDH} for solving SPDEs driven by L\'evy noise in bounded domains in \(\mathbb{R}^d\). For example, in \cite{CDH}, it was proved that for any fixed \( T > 0 \) and $D \in \mathfrak{B}$, the stochastic heat equation on \( D \),
\begin{equation}
    \label{heat-bdd}
    \begin{cases}
    \dfrac{\partial u}{\partial t} (t,x) = \dfrac{1}{2} \Delta u (t,x) + \sigma(u(t,x)) \dot{\Lambda} (t,x), \quad &  t \in (0,T], \, x \in D, \\
    u(t,x) = 0, \quad &   t \in [0,T], \, x \in \partial D, \\
    u(0,x) = u_0(x), \quad &  x \in D,
    \end{cases}
\end{equation}
has a unique solution \( u = \{ u(t,x) \; ; \; t \in [0,T], \, x \in D \} \) that satisfies
\[
\sup_{t \in [0,T]} \sup_{ x \in D } \mathbb{E} \left[ \left| u(t,x) \right|^p  \mathds{1}_{ {[\![ 0, \tau^{\star}_N(D)  ]\!]} }(t) \right] < +\infty,
\]
for all \( N \in \mathbb{N} \). Here, \( \tau^{\star}_N(D) \) is the stopping time defined by
\[
\tau^{\star}_N (D) := \inf \left\{ t \in [0,T] : \int_0^t \int_{D} \int_{\{ |z| > N \}} J (ds, dx, dz) > 0 \right\}.
\]
However, the main issue associated with the use of \( \tau^{\star}_N (D) \) for SPDEs on the entire space (such as \eqref{wave-levy} and \eqref{semilinear-heat}) is that if \( D = \mathbb{R}^d \), then \( \tau^{\star}_N (D) = 0 \) almost surely for all \( N \in \mathbb{N} \). This happens because the region \([0,t] \times \mathbb{R}^d \times \{ |z| > N \}\) may contain infinitely many points of \( J \). By contrast, using the PLCP, we can construct a "local solution" \( u^{(D)} \) of \eqref{wave-levy-T} on \([0,T] \times \overline{D}\) using the stopping times given by \eqref{stop-time-D}. As these local solutions are consistent and agree almost surely on the same region in space, we can construct a mild solution to \eqref{wave-levy-T}.

The resemblance between \( \tau_N (D) \) and \( \tau^{\star}_N (D) \) is not a coincidence. The primary result of this section shows that the solution \( u \) to \eqref{wave-levy-T} given by Theorem \ref{wave_D} satisfies \( u^{(D)}(t, x) = u(t, x) \) a.s. for all \( (t, x) \in [0, T] \times \overline{D} \), where \( u^{(D)} \) is a predictable process satisfying the stochastic-integral equation:
\begin{equation}
    \label{wave-bdd}
    u^{(D)}(t, x) = w(t, x) + \int_0^t \int_{\mathbb{R}^d} G_{t-s}(x - y) \sigma(u^{(D)}(s, y)) \mathds{1}_{\mathcal{K}^D (s)}(y) \Lambda(ds, dy).
\end{equation}
Note that the integrand of the stochastic integral on the right-hand side of equation \eqref{wave-bdd} has support on \( \mathcal{K}^D (0) \). Hence, we can solve \eqref{wave-bdd} using \( \tau_N (D) \) similarly to solving SPDEs driven by L\'evy noise in bounded domains.

Furthermore, since the method that we use to construct a solution to \eqref{wave-bdd} essentially requires the same integrability condition on \( \nu \) as for solving SPDEs in bounded domains, it suffices to consider only the assumption on the small jumps \( \int_{\{|z| \le 1\}} |z|^p \nu (dz) < +\infty \) for some \( p >0 \). More precisely, we only need this assumption on the small jumps for the wave equation \eqref{wave-levy-T} in dimension \( d = 2 \), while for dimension \( d = 1 \), no additional conditions are required beyond \eqref{levy-measure0}.

\begin{theorem}
    \label{wave_D}
   (a) If \( d = 1 \), \eqref{wave-levy-T} has a unique  (up to modifications) mild solution \( u \) that satisfies
    \[
    \sup_{t \in [0, T]} \sup_{x \in \overline{D}} \mathbb{E} \left[ |u(t, x) |^p \mathds{1}_{ {[\![ 0, \tau_N (D) ]\!]} } (t) \right] < +\infty, \quad \text{for all \( p \ge 2 \), \( N \in \mathbb{N} \), and \( D \in \mathfrak{B} \)}.
    \]
    (b) For \(d=2\), there exists \(p \in (0,2)\) such that
    \begin{equation}
       \label{levyp}
      \int_{\{|z| \le 1\}} |z|^p \nu (dz) <+ \infty,
    \end{equation}
    and assume that \(b = \int_{\{ |z| \le 1 \}} z \nu(dz)\) if \(p < 1\). Then, \eqref{wave-levy} has a unique (up to modifications) mild solution \(u\) that satisfies
    \[
    \sup_{t \in [0, T]} \sup_{x \in \overline{D}} \mathbb{E} \left[ \left| u(t, x) \right|^p \mathds{1}_{ {[\![ 0, \tau_N (D) ]\!]} } (t) \right] < +\infty, \quad \text{for all } N \in \mathbb{N} \text{ and } D \in \mathfrak{B}.
    \]
\end{theorem}

\begin{proof} 
As in the previous section, the main distinction between \(d=1\) and \(d=2\) lies in the integrability properties of \(G_t\). If \(d=1\), \(\int_{0}^T \int_{\mathbb{R}^d} G_t^p (x) \, dt \, dx < +\infty\) for all \(p > 0\), and \(\int_{\{|z| \le 1\}} |z|^p \nu(dz) <+ \infty\) for all \(p \ge 2\) by Remark \ref{d1-small}. In contrast, if \(d=2\), \(\int_{0}^T \int_{\mathbb{R}^d} G_t^p (x) \, dt \, dx <+ \infty\) holds only for \(p \in (0, 2)\), which requires imposing the \(p\)-integrability of \(\nu\) on the small jumps.

\textit{Step 1 (Existence and uniqueness of a local solution).} Let \( D \in \mathfrak{B} \). By employing a similar approach as in the proofs of Theorem \ref{wave-jimenez0} and Theorem \ref{wave-jimenez1}, we can establish the existence of a unique solution \( u^{(D)} \) to \eqref{wave-bdd} that satisfies
\[
\sup_{t \in [0, T]} \sup_{|x| \le R} \mathbb{E} \left[ \left| u^{(D)}(t, x) \right|^p \mathds{1}_{ {[\![ 0, \tau_N(D)  ]\!]} }(t) \right] < +\infty, \quad \text{for all } N \in \mathbb{N} \text{ and } R > 0.
\]
Additionally,
\[
u^{(D)}(t, x) = u^{(D, N)}(t, x), \quad \text{on} \quad \{ t \le \tau_N (D) \},
\]
where \( u^{(D, N)}(t, x) \), up to modifications, is the unique solution to the truncated problem, 
\begin{equation}
    \label{wave-bdd-trunc}
    u^{(D,N)}(t,x) = w(t,x) + \int_0^t \int_{\mathbb{R}^d} G_{t-s} (x - y) \sigma(u^{(D,N)}(t,x))  \mathds{1}_{ \mathcal{K}^D (s)} (y) \overline{\Lambda}_N(ds,dy),
\end{equation}
that satisfies
\begin{equation}
\label{Bploc-bdd}
\sup_{t \in [0, T]} \sup_{|x| \le R} \mathbb{E} \left[  \left| u^{(D,N)}(t, x) \right|^p \right]  < +\infty, \quad \text{for all} \; R > 0.
\end{equation}
Here, $\overline{\Lambda}_N $ is the noise $\Lambda_N$ in \eqref{levy-noise0} when $h(x)=1$, i.e.,
\begin{equation}
    \label{truncate-classic}
         \overline{\Lambda}_N (B) =  b |B |+  \int_{ B \times \{ |z| \le 1 \} } z  \tilde{J}(dt,dx,dz) + \int_{ B \times \{ 1 < |z| \le N   \}  } z  J(dt,dx,dz).
\end{equation}
Notice that \( u^{(D,N)}(t,x) \) is the limit in $B^p_{\text{loc}}$ of the Picard iteration sequence \( \{ u_n^{(D,N)} \}_{n \ge 0} \), given by \( u_0^{(D,N)} = \Psi_0^{(D)} \), with \( \Psi_0^{(D)} \) being an arbitrary element of \( B^p_{\text{loc}} \), and 
\begin{equation}
\label{picard-bdd}
 u_n^{(D,N)} (t,x) = w(t, x) + \int_0^t \int_{\mathbb{R}^d} G_{t-s}(x-y) \sigma(u_{n-1}^{(D, N)}(s, y)) \mathds{1}_{ \mathcal{K}^D (s)}(y) \overline{\Lambda}_N(ds, dy), 
\end{equation}
for $n \in \mathbb{N}.$

The primary distinction from the proofs of Theorem \ref{wave-jimenez0} and Theorem \ref{wave-jimenez1} is that we can apply Lemma \ref{lem33chong} for \( \overline{\Lambda}_N \) ( $\Lambda_N$ when \( h(x)=1 \)), without condition \eqref{large-jumps-sum}, i.e., considering only condition \eqref{levyp}; this same observation can be found on page 477 of \cite{balan49}. This enables us to demonstrate that
\begin{equation}
\label{conv-D-trunc}
\sup_{t \in [0,T]} \sup_{ |x| \le R } \mathbb{E} \left[ \left| u_n^{(D,N)}(t,x) - u^{(D,N)}(t,x) \right|^p \right]  \to 0, \quad \text{for all  $R>0,$}
\end{equation}
as \( n \to + \infty \), without \eqref{large-jumps-sum}. The uniqueness of the solution $u^{(D)}$ can be proven using the same arguments of the proof of Theorem \ref{wave-jimenez1}.

The proof for constructing local-truncated solutions \( u^{(D,N)} \) for \( D \in \mathfrak{B} \) is essentially the same as in Theorem 2.8 of \cite{balan49}; the only differences are that the initial condition \(w\) satisfies \eqref{ICH-eq} and we used the compact support property \eqref{compact_supp} to obtain the self-map property of the truncated operator related to the fixed-point problem \eqref{wave-bdd-trunc}. This implies that we can obtain a unique (up to modifications) mild solution to \eqref{wave-bdd-trunc} that satisfies \eqref{Bploc-bdd}.

\textit{Step 2 (Consistency).} For any \( A, B \in \mathfrak{B} \) with \( A \subset B \), we will show:
\begin{equation}
    \label{contain1}
    u^{(A)}(t,x) = u^{(B)}(t,x) \quad \text{a.s.,} \quad \text{for all } (t,x) \in [0,T] \times \overline{A},
\end{equation}   
where \( u^{(A)} \) (resp. \( u^{(B)} \)) is the solution of \eqref{wave-bdd} found in Step 1 with \( D = A \) (resp. \( D = B \)).

Our goal is to demonstrate that
\begin{equation}
    \label{goal1}
   \sup_{(t,x) \in [0, T] \times \overline{A} } \mathbb{E} \left[ | ( u^{(A)}(t,x) - u^{(B)}(t,x)) \mathds{1}_{[\![ 0, \tau_N (B) \wedge \tau_N (A) ]\!]}(t) |^p \right] = 0 \quad \text{for all } N \in \mathbb{N}.
\end{equation}
If \eqref{goal1} holds, then,
\[
  \mathds{1}_{[\![ 0, \tau_N (B) \wedge \tau_N (A) ]\!]} (t) u^{(A)} (t,x) = \mathds{1}_{[\![ 0, \tau_N (B) \wedge \tau_N (A) ]\!]} (t) u^{(B)} (t,x) \quad \text{a.s.} \quad \text{for all } (t,x) \in [0, T] \times \overline{A}.
\]
Letting \(N\) be sufficiently large, we obtain \eqref{contain1}. Now, for any \( (t,x) \in [0, T] \times \overline{A} \), by the triangle inequality, we have:
\begin{equation}
\label{tri1}
\begin{split}
   &  \mathbb{E} \left[ |(u^{(A)}(t,x) - u^{(B)}(t,x)) \mathds{1}_{[\![0, \tau_N(B) \wedge \tau_N(A)]\!]}(t)|^p \right] \\ 
    & \le c_p \Big[\mathbb{E} \left[ |(u^{(A,N)}(t,x) - u_n^{(A,N)}(t,x)) \mathds{1}_{[\![0, \tau_N(B)]\!]}(t)|^p \right] \\
    &+ \mathbb{E} \left[ |(u_n^{(A,N)}(t,x) - u_n^{(B,N)}(t,x)) \mathds{1}_{[\![0, \tau_N(B)]\!]}(t)|^p \right]  \\
   &  + \mathbb{E} \left[ |(u_n^{(B,N)}(t,x) - u^{(B,N)}(t,x)) \mathds{1}_{[\![0, \tau_N(B)]\!]}(t)|^p \right] \Big], \\
\end{split}
\end{equation}
where $u_n^{(A,N)}$ (resp. \( u_n^{(B,N)} \)) is the sequence defined in \eqref{picard-bdd} when $D=A$ (resp.  $D=B$). In \eqref{tri1}, we used the fact that \(\tau_N (B) \le \tau_N(A)\) a.s. for all \(N \in \mathbb{N}\). By \eqref{conv-D-trunc}, the first and third terms on the right-hand side of \eqref{tri1} converge to zero. Therefore, to prove \eqref{goal1}, it remains to show that:
\begin{equation}
    \label{goal2}
    \sup_{(t,x) \in [0, T] \times \overline{A}} \mathbb{E} \left[ |(u_n^{(A,N)}(t,x) - u_n^{(B,N)}(t,x)) \mathds{1}_{[\![0, \tau_N(B)]\!]}(t)|^p \right] \to 0 \quad \text{as } n \to + \infty.
\end{equation}

For a fixed  $(t,x) \in [0,T] \times \overline{A}$ and fixed $n \in \mathbb{N}$, we define the set $\mathcal{A}_n^{(t,x)}$ as
\[
\begin{split}
    \mathcal{A}_n^{(t,x)} := \Big\{ (t,x,s_1,y_1, & \ldots, s_{n}, y_{n})  \in  ([0,T] \times \mathbb{R}^d)^{n+1} ; \\
    &\, t>s_1 > s_2 > \ldots > s_n >0, \, y_{k} \in B_{ s_{k-1} - s_{k} } (y_{k-1}), \, k=1, \ldots,n \Big\},\\
\end{split}
\]
where, $(s_0,y_0) = (t,x)$. In contrast with the previous section, we use the reverse order index for the simplex defined in Lemma \ref{lemma35}, i.e., \( t > s_1 > s_2 > \ldots > s_n > 0 \). We denote
\[
\mathbf{s}:=(s_1, s_2, \ldots, s_n), \quad \text{and} \quad \mathbf{y}:= (y_1, y_2, \ldots, y_n).
\]
Additionally, note that
\begin{equation}
\label{A-reverse}
\int_{\tilde{T}_n(t)} \int_{(\mathbb{R}^d)^n}
 \prod_{i=1}^n g_p(s_{i-1} - s_{i}, y_{i-1} - y_{i}) \, d \mathbf{y} \, d \mathbf{s} = A_n^{(p)}(t),    
\end{equation}
where \( \tilde{T}_n(t) := \{ (s_1, \ldots, s_n) \in (0,t)^n \, ; \, s_n < \ldots < s_1 \} \), and $A_n^{(p)}(t)$ is defined as in \eqref{sup_AA}.

By the triangle inequality, note that
\begin{equation}
\label{chain-ball}
B_{s_{i-1}-s_i} (y_{i-1}) \subset   \mathcal{K}^{A} (s_i) \subset \mathcal{K}^{B} (s_i), \quad \text{for all $i=1,\ldots,n$}.
\end{equation}
for any $ (t,x,s_1,y_1, \ldots, s_{n}, y_{n}) \in \mathcal{A}_n^{(t,x)}$. Due to \eqref{chain-ball}, we have:
\begin{equation}
\label{ball-picardA}
\begin{split}
     u^{(A,N)}_{n-k+1} (s_{k-1},y_{k-1})
     & =  w(s_{k-1},y_{k-1}) \\
     & + \int_0^{s_{k-1}} \int_{\bR^d} G_{s_{k-1}-s_{k}} (y_{k-1}-y_{k}) \sigma( u^{(A,N)}_{n-k} (s_{k},y_{k}) )  \overline{\Lambda}_N (ds_k,dy_k) \\
\end{split}
\end{equation}
and
\begin{equation}
\label{ball-picardB}
\begin{split}
     u^{(B,N)}_{n-k+1} (s_{k-1},y_{k-1})
     & = w(s_{k-1},y_{k-1}) \\
      & + \int_0^{s_{k-1}} \int_{\bR^d}  G_{s_{k-1}-s_{k}}(y_{k-1}-y_{k}) \sigma( u^{(B,N)}_{n-k} (s_k,y_k) ) \overline{\Lambda}_N (ds_k,dy_k). \\
\end{split}
\end{equation}
for $k=1,\ldots,n.$

Let
\[
Y_{n,k}^N (s_{k},y_{k}) := (u^{(A,N)}_{n-k} (s_{k},y_{k}) -   u^{(B,N)}_{n-k} (s_{k},y_{k}) ) \mathds{1}_{[\![  0, \tau_N (B)   ]\!]} (s_{k} ), 
\]
for all $k=0,1,\ldots,n.$ Then, using \eqref{ball-picardA}, \eqref{ball-picardB}, Proposition \ref{local-int}, and Lemma \ref{lem33chong}, we have:
\begin{equation}
   \label{2_it}
    \begin{split}
       & \mathbb{E} \left[ |Y_{n,k-1}^N (s_{k-1}, y_{k-1}) |^p \right] \\
       & \le C_T \int_0^t \int_{\mathbb{R}^d} g_p (s_{k-1} - s_{k}, y_{k-1} - y_{k}) \mathbb{E} \left[ |Y_{n,k}^N (s_{k}, y_{k}) |^p \right]  \, dy_k \, ds_k, \\
    \end{split}
\end{equation}
for \( k = 1, \ldots, n \). Observe that the constant \(C_T\) in \eqref{2_it} differs from the constant in \eqref{bdgl2}. However, this is not an issue as this constant depends only on \(N\), \(T\), \(\sigma\), and \(p\).

Now, by \eqref{2_it}, we can iterate the following inequality \( n - 1 \)-times:
\begin{equation}
\label{1_it}
    \begin{split}
       & \mathbb{E} \left[ |Y_{n,0}^N (t,x)|^p \right]  \le C_T \int_0^t \int_{\mathbb{R}^d} g_p(t-s_1, x-y_1) \mathbb{E} \left[ |Y_{n,1}^N (s_1,y_1) |^p \right]  \, dy_1 \, ds_1,
    \end{split}
\end{equation}
over \( (t,x,s_1,y_1,\ldots,s_n,y_n) \in \mathcal{A}^{(t,x)} \). Thus,
\begin{equation}
\label{iterati1}
  \mathbb{E} \left[ |Y_{n,0}^N (t,x)|^p \right] \le  C_T^n \int_{\tilde{T}_n(t) } \int_{(\mathbb{R}^d)^n}
\prod_{i=1}^n  g_p( s_{i-1} - s_{i}, y_{i-1} - y_{i} )   \bE \left[ |Y_{n,n}^N (s_n,y_n) |^p \right]  d \textbf{y} d \textbf{s}.
\end{equation}
Therefore, by \eqref{iterati1}, \eqref{A-reverse}, \eqref{A_1_SUP} and \eqref{A_22_SUP},
\[
\begin{split}
    & \bE \left[ |    u^{(A,N)}_n (t,x) -   u^{(B,N)}_n (t,x) \mathds{1}_{[\![  0, \tau_N (B)   ]\!]} (t) |^p \right] \\
    & \le    C_T^n \int_{ \tilde{T}_n(t) } \int_{(\mathbb{R}^d)^n}
\prod_{i=1}^n  g_p( s_{i-1} - s_{i}, y_{i-1} - y_{i} )   \bE \left[ |u_{0}^{(A,N)} (t_n,x_n) - u_{0}^{(B,N)} (s_n,y_n) |^p \right]  d \textbf{y} d \textbf{s} \\ 
  & \le \sup_{s \in [0,T]} \sup_{|y| \le R_A +T}  \bE \left[ |u_{0}^{(A,N)} (s,y) - u_{0}^{(B,N)} (s,y) |^p  \right] \ \sup_{t \in [0,T]}  C_T^n  A_n^{(p)} (t)  \to 0, \quad \text{as $n \to + \infty$,} \\ 
\end{split}
\]
where \( R_A := \sup_{ x \in \overline{A}} |x| \).

\textit{Step 3 (Global solution).}
Let \( u = \{ u(t,x) \, ; \, t \ge 0, x \in \mathbb{R}^d \} \) be the random field defined by
\begin{equation}
    \label{global}
    u(t,x) := u^{(D_k)} (t,x), \quad \text{if} \quad (t,x) \in [0,T] \times [-k,k]^d,
\end{equation}
where \( D_k = (-k,k)^d \) for all \( k \in \mathbb{N} \). Here, \( u^{(D_k)} (t,x) \) is the solution to \eqref{wave-bdd} when \( D = D_k \). Note that \( u \) is well-defined due to \eqref{contain1}. Furthermore, \( u \) defined in \eqref{global} is a solution to \eqref{wave-levy-T}. To demonstrate this, let us fix \( (t,x) \in [0,T] \times \mathbb{R}^d \), and define \( V_{T,x} = \{ y \in \mathbb{R}^d ; |y| < T + |x| \} \). Then, we have:
\[
\begin{split}
    &w(t,x) + \int_0^t \int_{\mathbb{R}^d} G_{t-s} (x-y) \sigma(u(s,y)) \Lambda(ds,dy) \\
    &= w(t,x) + \int_0^t \int_{\mathbb{R}^d} G_{t-s} (x-y) \sigma(u^{(V_{T,x})}(s,y)) \mathds{1}_{\mathcal{K}^{V_{T,x}}(s)}(y) \Lambda(ds,dy) \\
    &= u^{(V_{T,x})} (t,x) = u(t,x) \quad \text{a.s.}
\end{split}
\]
This shows that \( u \) is indeed a solution to \eqref{wave-levy-T}.

\textit{Step 4 (Uniqueness of the solution).}
We assume that $v$ is a mild solution of \eqref{wave-levy-T} that satisfies
\begin{equation}
\label{unii}
\sup_{t \in [0,T]} \sup_{x \in \overline{D}} \bE \left[ |v(t,x)|^p \mathds{1}_{ {[\![ 0, \tau_N (D) ]\!]} }(t) \right]  < + \infty,    
\end{equation}
for all $N \in \bN$ and $D \in \mathfrak{B}$, where $\tau_N (D)$ is given by \eqref{stop-time}. We will show that for any $D$:
\[
u(t,x) = v(t,x) \quad \text{a.s.}  \quad \text{for all $t \in [0,T]$ and $x \in \overline{D}$,}
\]
where $u$ is the solution to \eqref{wave-levy-T} on the previous step. 

First, we define
\[
V_{T,D}:= \{ y \in \bR^d ; |y| < T + R_D \}, \quad \text{where $R_D = \sup_{ x \in \overline{D}} |x|$}.
\]
Note that for all $(t,x) \in [0,T] \times \overline{D}$, we have:
\[
u(t,x) = u^{(V_{T,D})} (t,x) \quad \text{a.s.}
\]
Also, recall that \( u^{(V_{T,D})}(t,x)=u^{(V_{T,D},N)}(t,x) \) a.s. on the event \( \{t\leq \tau_N(V_{T,D})\}\) (see Step 1). Hence, by the triangle inequality,
\[
\begin{split}
& \sup_{t \in [0,T]} \sup_{x \in \overline{D}} \bE \left[ |(u(t,x)-v(t,x)) \mathds{1}_{ {[\![ 0, \tau_N (V_{T,D}) ]\!]} }(t) |^p  \right] \\
& \le 2^{p-1} \Bigg( \sup_{t \in [0,T]} \sup_{x \in \overline{D}} \bE \left[ |(u^{(V_{T,D},N)}(t,x)-u_n^{(V_{T,D},N)}(t,x)) \mathds{1}_{ {[\![ 0, \tau_N (V_{T,D}) ]\!]} }(t) |^p \right] \\
& + \sup_{t \in [0,T]} \sup_{x \in \overline{D}} \bE \left[ |(u_n^{(V_{T,D},N)}(t,x) - v(t,x)) \mathds{1}_{ {[\![ 0, \tau_N (V_{T,D}) ]\!]} }(t) |^p \right] \Bigg),\\
\end{split}
\]
where $u_n^{(V_{T,D},N)}$ is the sequence in \eqref{picard-bdd} with respect to $V_{T,D}$. The first term on the right hand side of the inequality above converges to 0 by \eqref{conv-D-trunc}. It remains to show that:
\[
\sup_{t \in [0,T]} \sup_{x \in \overline{D}} \bE \left[ |(u_n^{(V_{T,D},N)}(t,x) - v(t,x)) \mathds{1}_{ {[\![ 0, \tau_N (V_{T,D}) ]\!]} }(t) |^p \right] \to 0 \quad \text{as $n \to + \infty$,}
\]
for a fixed $N \in \bN$.

Let $(t, x, s_1, y_1, \ldots, s_n, y_n) \in \mathcal{A}^{(t,x)}$, using the same argument as in \eqref{chain-ball}, we get:
\begin{equation}
\label{chain-ball1}
B_{s_{i-1}-s_i} (y_{i-1}) \subset   \mathcal{K}^{D} (s_i) \subset \mathcal{K}^{V_{T,D}} (s_i), \quad \text{for all $i=1,\ldots,n$}.
\end{equation}
Then, by \eqref{chain-ball1} and \eqref{local_D}, we have:
\begin{equation}
\label{it_3}   
\begin{split}
     & \mathds{1}_{ {[\![  0, \tau_N (V_{T,D}) ]\!]} } (s_{k-1})  \int_0^{s_{k-1}} \int_{\bR^d} G_{s_{k-1} -s_k}(y_{k-1}-y_k ) \sigma ( v (s_k,y_k) )  \Lambda (ds_k,dy_k)    \\
     = & \mathds{1}_{ {[\![  0, \tau_N (V_{T,D}) ]\!]} }  (s_{k-1})  \\
    &  \times  \int_0^{s_{k-1}} \int_{\bR^d} G_{s_{k-1} -s_k}(y_{k-1}-y_k ) \sigma ( v (s_k,y_k) )  \mathds{1}_{ {[\![  0, \tau_N (V_{T,D}) ]\!]} } (s_k)  \overline{\Lambda}_N (ds_k,dy_k), \\
\end{split}
\end{equation}
for \( k = 1, \ldots, n. \) 

Now, let us define,
\[
W^N_{n,k} (s_{k}, y_{k}) := (u^{(V_{T,D}, N)}_{n-k} (s_{k}, y_{k}) - v(s_{k}, y_{k})) \mathds{1}_{[\![ 0, \tau_N (V_{T,D}) ]\!]} (s_{k}),
\]
for \( k = 1, \ldots, n. \) Using \eqref{chain-ball}, \eqref{it_3}, \eqref{local_D}, and Lemma \ref{lem33chong}, we obtain:

\begin{equation}
   \label{5_it}
    \begin{split}
       & \bE \left[ |W_{n,k-1}^N (s_{k-1},y_{k-1}) |^p \right] \\
       & \le  C_T \int_0^t \int_{\mathbb{R}^d} g_p (s_{k-1}-s_{k}, y_{k-1}-y_{k}) \mathbb{E} \left[ |W_{n,k}^N (s_{k},y_{k}) |^p \right] \, dy_k \, ds_k, \\
    \end{split}
\end{equation}
for $k=1, \ldots, n.$ Applying the same reasoning as in Step 2, we get:
\begin{equation}
\label{iterati22}
  \mathbb{E} \left[ |W_{n,0}^N (t,x)|^p \right] \le  C_T^n \int_{\tilde{T}_n(t) } \int_{(\mathbb{R}^d)^n}
\prod_{i=1}^n  g_p( s_{i-1} - s_{i}, y_{i-1} - y_{i} )   \bE \left[ |W_{n,n}^N (s_n,y_n) |^p \right] d \textbf{y} d \textbf{s}.
\end{equation}
Therefore, by \eqref{unii}, \eqref{A-reverse}, \eqref{A_1_SUP} and \eqref{A_22_SUP},
\[
\begin{split}
&  \sup_{t \in [0,T]} \sup_{x \in \overline{D}} \bE \left[ |(u_n^{(V_{T,D},N)}(t,x) - v(t,x)) \mathds{1}_{ {[\![ 0, \tau_N (V_{T,D}) ]\!]} }(t) |^p \right] \\
& \le \sup_{s \in [0,T]} \sup_{|y| \le R_D + T} \bE \left[ |(u_{0}^{(V_{T,D},N)}(s,y) - v(s,y)) \mathds{1}_{ {[\![ 0, \tau_N (V_{T,D}) ]\!]} }(s) |^p \right] \sup_{t \in [0,T]} C_T^n A_n^{(p)}(t) \to 0,
\end{split}
\]
as $n \to + \infty$. Finally, we conclude that
\[
u(t,x)  \mathds{1}_{ {[\![ 0, \tau_N (V_{T,D}) ]\!]} }(t) = \mathds{1}_{ {[\![ 0, \tau_N (V_{T,D}) ]\!]} }(t) v(t,x) \quad \text{a.s.}, \quad \text{for all $(t,x) \in [0,T] \times \overline{D}$}.
\]
For sufficiently large \( N \in \mathbb{N} \), we have \(\tau_N (V_{T,D}) = +\infty\) a.s. Consequently, \(u(t,x) = v(t,x)\) a.s., for any \((t,x) \in [0,T] \times \overline{D}\).
\end{proof}

The following remark is derived directly from the proof of Theorem \ref{wave_D}.
\begin{remark}
Let \( u \) be the solution to \eqref{wave-levy-T} given by Theorem \ref{wave_D}. Then,
\[
u(t,x) = u^{(D)} (t,x) \quad \text{a.s.,} \quad \text{for all } (t,x) \in [0,T] \times \overline{D},
\]
where \( u^{(D)} \) is the solution to \eqref{wave-bdd} constructed in Step 1 of Theorem \ref{wave_D}.
\end{remark}

A natural inquiry at this point is the relationship between the solution obtained in Theorem \ref{wave_D} for a fixed time interval $[0, T]$, and the solution to \eqref{wave-levy} derived in Theorem \ref{wave-jimenez1}, under identical initial conditions and same function \( \sigma \). The following theorem confirms that both solutions are almost surely identical for all \( (t,x) \in [0, T] \times \mathbb{R}^d \).

\begin{theorem}
\label{time-comp}
Under the same assumptions as in Theorem \ref{wave_D}, if \( v \) is a mild solution to \eqref{wave-levy-T} that satisfies
\[
\sup_{t \in [0,T]} \sup_{|x| \le R} \mathbb{E} \left[ \left| v(t,x) \right|^p \mathds{1}_{ {[\![ 0, \tilde{\tau}_N  ]\!]} }(t) \right]  < +\infty, \quad \text{for all } N \in \mathbb{N} \text{ and } R > 0,
\]
where \(p\) is the same exponent as in Theorem \ref{wave_D}, \(\tilde{\tau}_N\) is a non-decreasing sequence of stopping times such that \(\tilde{\tau}_N \uparrow +\infty \) a.s. for \( N \to +\infty \), then \( v(t,x) = u(t,x) \) a.s. for all \( (t,x) \in [0,T] \times \mathbb{R}^d \), where \( u \) is the solution to \eqref{wave-levy-T} given by Theorem \ref{wave_D}.
\end{theorem}

\begin{proof}
Let \( V_{T,R} := \{ y \in \mathbb{R}^d \; ; \; |y| < T + R \} \) for a fixed \( R > 0 \). Note that \( u(t,x) = u^{(V_{T,R})} (t,x) \) for any \( t \in [0,T] \) and \( |x| \le R \). Hence, by the triangle inequality,
\[
\begin{split}
& \sup_{t \in [0,T]} \sup_{|x| \le R} \mathbb{E} \left[ \left| (u(t,x) - v(t,x)) \mathds{1}_{ {[\![ 0, \tilde{\tau}_N \wedge \tau_N (V_{T,R}) ]\!]} }(t) \right|^p \right]  \\
& \le 2^{p-1} \Bigg( \sup_{t \in [0,T]} \sup_{|x| \le R} \mathbb{E} \left[ \left| (u^{(V_{T,R},N)}(t,x) - u_n^{(V_{T,R},N)}(t,x)) \mathds{1}_{ {[\![ 0, \tilde{\tau}_N \wedge \tau_N (V_{T,R}) ]\!]} }(t) \right|^p \right] \\
& \quad + \sup_{t \in [0,T]} \sup_{|x| \le R} \mathbb{E} \left[ \left| (u_n^{(V_{T,R},N)}(t,x) - v(t,x)) \mathds{1}_{ {[\![ 0, \tilde{\tau}_N \wedge \tau_N (V_{T,R}) ]\!]} }(t) \right|^p \right] \Bigg). \\
\end{split}
\]
The first term on the right-hand side of the inequality above converges to 0. For the second term, using the same arguments as in Step 4 of the proof of Theorem \ref{wave_D}, it can be shown that this term also converges to 0. Consequently, \( u(t,x) = v(t,x) \) a.s. for all  \((t,x) \in [0,T] \times \mathbb{R}^d\).
\end{proof}

\begin{remark}
Theorem \ref{time-comp} implies that, under Assumption \ref{ICH}, Assumption \ref{wave-levy-assumption}, and with the same initial conditions and function \( \sigma \), the solution to \eqref{wave-levy-T} given by Theorem \ref{wave_D} is almost surely identical to the solution to \eqref{wave-levy} given by Theorem \ref{wave-jimenez1} on \( [0,T] \times \mathbb{R}^d \).
\end{remark}

\appendix

\section{Stochastic integration}
For the reader's convenience, this section presents some basic elements of the theory of stochastic integration of \(L^p\)-random measures (see \cite{bit1,bit2}), which are required to perform stochastic calculus with respect to heavy-tailed noise. There are different ways to define random measures. In this work, we follow the definition of \(L^p\)-random measures presented in \cite{chong0}.

\begin{definition}[Definition 2.1 in \cite{chong0}]
\label{LPRM}
Let $\{ \Omega_k \}_{k \in \bN}$ be a sequence of sets in $\cP$ satisfying $\Omega_k \uparrow \tilde{\Omega}$. A map 
\[
M: \cP_M = \bigcup_{k \ge 1} \cP \Big|_{\Omega_k} \to L^p ( \Omega )
\]
is called an \underline{$L^p$-random measure} in $\tilde{\Omega}$ if it satisfies:
\begin{itemize}
    \item[i)] $M( \emptyset ) = 0$ a.s.,
    \item[ii)] For every sequence $\{ A_i \}_{i \in \bN}$ of pairwise disjoint sets in $\cP_M$ with $\bigcup_{i=1}^\infty A_i \in \cP_M$, we have
    \[
     M \Big( \bigcup_{i=1}^{\infty} A_i \Big) = \sum_{i=1}^\infty M(A_i ) \qquad \text{in } L^p(\Omega).
    \]
    \item[iii)] For all $A \in \cP_M$ with $A \subset \Omega \times (t, +\infty) \times \bR^d$ for some $t \in \bR$, the random variable $M(A)$ is $\mathcal{F}_t$-measurable.
    \item[iv)] For all $A \in \cP_M$, $t \in \bR$, and $F \in \cF_t$, we have
    \[
      M(A \cap (F \times (t, +\infty) \times \bR^d )) = \mathds{1}_F M(A \cap (\Omega \times (t, +\infty) \times \bR^d)) \quad a.s.
    \]
\end{itemize}
\end{definition}

In \cite{RR}, a \textit{L\'evy basis} is defined as an infinitely divisible independently scattered random measure on \(\mathbb{R}_+ \times \mathbb{R}^d\). In \cite{chong0,chong1}, the definition of \textit{L\'evy basis} has been modified to incorporate the theory of stochastic integration developed in \cite{bit1}. In this work, we use the same definition of \textit{L\'evy basis} as in \cite{chong0,chong1}. We recall this definition below.

\begin{definition}
\label{levy-basis-def}
An \(L^0\)-random measure \(\tilde{\Lambda}: \tilde{\mathcal{P}_b} \to L^0\) is called a \underline{L\'evy basis} if it satisfies:
\begin{itemize}
    \item[i)]  Let \(\{ B_i \}_{i \in \mathbb{N}}\) be a sequence of pairwise disjoint sets in \(\mathcal{B}_b (\mathbb{R}_+ \times \mathbb{R}^d)\), then \(\{ \tilde{\Lambda} (\Omega \times B_i) \}_{i \in \mathbb{N}}\) is a sequence of independent random variables. Additionally, if \(B \in \mathcal{B}_b (\mathbb{R}_+ \times \mathbb{R}^d)\) satisfies \(B \subset (t, +\infty) \times \mathbb{R}^d\) for some \(t \in \mathbb{R}_+\), then \(\tilde{\Lambda} (\Omega \times B)\) is independent of \(\mathcal{F}_t\).
    \item[ii)] For all \(B \in \mathcal{B}_b (\mathbb{R}_+ \times \mathbb{R}^d)\), \(\tilde{\Lambda} (\Omega \times B)\) has an infinitely divisible distribution.
    \item[iii)] For all \(t \in \mathbb{R}_+\) and \(k \in \mathbb{N}\), we have \(\tilde{\Lambda}(\Omega \times \{t \} \times [-k,k]^d )=0\) a.s.
\end{itemize}
\end{definition}

Now, we extend the notion of pure-jump L\'evy white noise \(\Lambda = \{ \Lambda(B); B \in \mathcal{B}_b(\mathbb{R}_+ \times \mathbb{R}^d) \}\) defined for non-random Borel sets in \eqref{levy-noise1} to a L\'evy basis \(\tilde{\Lambda}\) in order to perform stochastic integration with respect to \(\tilde{\Lambda}\). The details of this extension can be found in \cite{chong0, HM, JS}.

\begin{remark}
\label{ext_levy}
The noise $\Lambda$  in \eqref{levy-noise1} has an extension $\tilde{\Lambda} : \tilde{\cP}_b \to L^0$, that satisfies
\begin{equation}
\label{extension_levy}
\tilde{\Lambda } (\Omega \times B) = \Lambda (B) \qquad \text{for all $B \in \cB_b ( \bR_+ \times \bR^d )$}.
\end{equation}
This random measure $\tilde{\Lambda}$ satisfies Definition \ref{levy-basis-def}.
\end{remark}

Let \(\mathcal{S}\) be the set of simple predictable processes of the form \(S = \sum_{i=1}^k \alpha_i \mathds{1}_{A_i}\) for some \(\alpha_i \in \mathbb{R}\) and \(A_i \in \tilde{\mathcal{P}}_M\). Then, the stochastic integral of \(S\) with respect to \(M\) is given by:
\[
I^{\Lambda}(S) = \int_0^\infty \int_{\mathbb{R}^d} S(t,x) M(dt,dx)= \sum_{i=1}^k \alpha_i M(A_i).
\]
For any \(p \ge 0\), the \textit{Daniell mean} of a process \(H = \{ H(t,x); t \ge 0, \, x \in \mathbb{R}^d \}\) with respect to an \(L^p\)-random measure \(M\) is defined by
\[
|| H ||_{M,p} = \sup_{S \in \mathcal{S}_\Lambda ; |S| \le |H|} || I^M (S) ||_{p} .
\]

A predictable process $H$ is said to be {\em $p$-integrable} with respect to $M$ if there exists a sequence $\{ S_n\}_{ n \ge 1} \subset \mathcal{S}$ such that $||S_n -H ||_{ M,p} \to 0$ as $n \to +\infty $. In this case, the {\em stochastic integral} of $H$ with respect to $M$ is given by:
\begin{equation}
    \label{def-int-levy}
    I^{M} (S) = \lim_{n \to +\infty } I^{M}(S_n) \quad \text{in} \quad L^p ( \Omega).
\end{equation}

We define the space
\[
L^{1,p}(M) = \Big\{ \text{the set of } p\text{-integrable processes with respect to } M \Big\}.
\]
In this definition, we omit writing \(p\) if \(p=0\). Note that the map \(I^M : L^{1,p}(M) \rightarrow L^p(\Omega)\) is a contraction. We include the following version of the dominated convergence theorem for stochastic integrals with respect to \(M\).

\begin{theorem}[Theorem A.1 of \cite{CDH}]
\label{DCT1}
If \(\{ H_n \}_{n \in \mathbb{N}}\) is a sequence of predictable processes that converges pointwise to \(H\), and \(|H_n| \le H_0\) for all \(n \ge 1\) and some \(H_0 \in L^{1,p}(M)\), then \(H_n, H \in L^{1,p}(M)\) and \(\|H - H_n\|_{M,p} \to 0\) as \(n \to +\infty\).
\end{theorem}

The following local property of the stochastic integral plays an important role in this article. While different versions of this property can be found in \cite{bit1,bit2,JS}, we include its proof here because we could not find a direct reference that addresses it with respect to a L\'evy basis as in Definition \ref{levy-basis-def}.

\begin{lemma}
     \label{local-int}
     Let $\tilde{\Lambda}$ be a L\'evy basis and $H(t,x)$ be a predictable process such that there exists an increasing sequence $\{T_n\}_{n \in \mathbb{N}}$ of stopping times with $T_N \uparrow +\infty $ a.s. for $N \to + \infty$, and $H(t,x) \mathds{1}_{[\![ 0, T_N]\!]} (t)$ is $p$-integrable with respect to $\tilde{\Lambda}$ for all $N \in \mathbb{N}$. Then, for any stopping time $\tau$ we have:
     \begin{equation}
         \label{local-prop-time}
               \mathds{1}_{[\![ 0, \tau]\!]} (t) \int_0^t \int_{\mathbb{R}^d} H(s,y) \tilde{\Lambda}(ds,dy) =  \mathds{1}_{[\![ 0, \tau]\!]} (t) \int_0^t \int_{\mathbb{R}^d} H(s,y) \mathds{1}_{[\![ 0, \tau]\!]} (s) \tilde{\Lambda}(ds,dy).
     \end{equation}
 \end{lemma}

\begin{proof}
\textit{Step 1:} first, we prove \eqref{local-prop-time} for an elementary stopping time $\tau$ and $H(s,y) = \mathds{1}_{A}(s,y)$, with $A \in \tilde{\cP}_b$, i.e.,
\begin{equation}
    \label{local-prop0}
    \mathds{1}_{[\![ 0, \tau]\!]} (t) \int_0^t \int_{\mathbb{R}^d} \mathds{1}_{A}(s,y) \tilde{\Lambda}(ds,dy)= \mathds{1}_{[\![ 0, \tau]\!]} (t) \int_0^t \int_{\mathbb{R}^d} \mathds{1}_{A}(s,y) \mathds{1}_{[\![ 0, \tau]\!]} (s) \tilde{\Lambda}(ds,dy).
\end{equation}
Let $\tau$ be an elementary stopping time, and assume that $\tau$ takes values
\[
\{ 0 = t_0 <  t_1 < t_2 < \ldots < t_{N_0 +1} \}.
\]
Then, the stochastic interval ${[\![  0, \tau]\!]} $ can be decomposed as
\begin{equation}
    \label{decomp_stochastic_int2}
    {[\![  0, \tau]\!]} =  (\{ \tau=0 \} \times \{ t = 0 \} ) \cup \bigcup_{n=0}^{N_0}  \{ \tau \ge t_{n+1} \} \times (t_n, t_{n+1} ].
\end{equation}
Hence,
\begin{equation}
    \label{decomp_stochastic_int1}
    \begin{split}
    \mathds{1}_{ {[\![  0, \tau]\!]} } &= \mathds{1}_{ \{ t=0 \} \times \{ \tau=0\} } + \sum_{n =0}^{N_0} \mathds{1}_{ \{ \tau \ge t_{n+1} \} \times (t_n, t_{n+1} ]}, \\
    \end{split}
\end{equation}
and observe that we can re-arrange the sum on \eqref{decomp_stochastic_int1} in the following way:
\begin{equation}
\label{decomp_stochastic_int1001}
\begin{split}
\sum_{n =0}^{N_0} \mathds{1}_{ \{ \tau \ge t_{n+1} \} \times (t_n, t_{n+1} ]} & =  \sum_{n =0}^{N_0} \sum_{ i= n +1 }^{N_0}  \mathds{1}_{ \{ \tau = t_{i} \} } \mathds{1}_{  (t_n, t_{n+1} ] } \\
& =  \sum_{n=1}^{N_0 + 1} \mathds{1}_{ \{ \tau = t_{n} \} \times (0, t_{n} ]}. \\ 
\end{split}   
\end{equation}
Then, we can re-write \eqref{decomp_stochastic_int1} as
\begin{equation}
    \label{decomp_stochastic_int12}
     \mathds{1}_{ {[\![  0, \tau]\!]} } =  \mathds{1}_{ \{ t=0 \} \times \{ \tau=0\} } +\sum_{n=1}^{N_0 + 1} \mathds{1}_{ \{ \tau = t_{n} \} \times (0, t_{n} ]}.
\end{equation}
Hence, by conditions (ii) and  (iv) of Definition \ref{LPRM} and Definition \ref{levy-basis-def}-(iii), we have:
\begin{equation}
\label{local001}
\begin{split}
     & \int_0^t \int_{\mathbb{R}^d} \mathds{1}_{A}(s,y)  \mathds{1}_{[\![  0, \tau]\!]} (s) \tilde{\Lambda}(ds,dy)  = \tilde{\Lambda}(A \cap ([\![  0, \tau]\!] \times \mathbb{R}^d  ) \cap (\Omega \times [0,t] \times \mathbb{R}^d)) \\
     &= \sum_{n = 0 }^{N_0} \tilde{\Lambda}( A \cap ( \{ \tau \ge t_{n+1} \} \times (t_n, t_{n+1} ] \times \mathbb{R}^d ) \cap (\Omega \times [0,t] \times \mathbb{R}^d)) \\
     & =  \sum_{n = 0 }^{N_0}  \mathds{1}_{  \{ \tau \ge t_{n+1} \} }  \tilde{\Lambda}( A \cap ( \Omega \times (t_n, t_{n+1} ] \times \mathbb{R}^d ) \cap (\Omega \times [0,t] \times \mathbb{R}^d) )\\
    & =  \sum_{n = 0 }^{N_0}  \mathds{1}_{  \{ \tau \ge t_{n+1} \} }  \int_{0}^{t} \int_{\mathbb{R}^d } \mathds{1}_{A}(s,y) \mathds{1}_{(t_n,t_{n+1}]} (s) \tilde{\Lambda}(ds,dy) \\
    & =  \sum_{n = 1 }^{N_0 +1}  \mathds{1}_{ \{ \tau= t_n \} }   \int_{0}^{t} \int_{\mathbb{R}^d } \mathds{1}_{A}(s,y) \mathds{1}_{(0,t_n]} (s)\tilde{\Lambda}(ds,dy). \\  
\end{split}
\end{equation}
Using the same argument as in \eqref{decomp_stochastic_int1001} and the linearity of $\tilde{\Lambda}$ on $\tilde{ \cP}_b$, we can re-arrange the sum in the last equality of \eqref{local001} as follows:
\begin{equation}
\label{decomp_stochastic_int102}
\begin{split}
    & \sum_{n = 0 }^{N_0}  \mathds{1}_{  \{ \tau \ge t_{n+1} \} } \int_{0}^{t} \int_{\mathbb{R}^d } \mathds{1}_{A}(s,y) \mathds{1}_{(t_n,t_{n+1}]} (s) \tilde{\Lambda}(ds,dy) \\ 
    & =  \sum_{n = 1 }^{N_0 +1}  \mathds{1}_{ \{ \tau= t_n \} } \int_{0}^{t} \int_{\mathbb{R}^d } \mathds{1}_{A}(s,y) \mathds{1}_{(0,t_n]} (s) \tilde{\Lambda}(ds,dy). \\
\end{split}
\end{equation}
Additionally, for a fixed \( i \in \{1,2, \ldots, N_0\} \), it holds:
\begin{equation}
    \label{localprop00}
    \mathds{1}_{ \{ \tau= t_i \} }  \mathds{1}_{[\![  0, \tau]\!]} (t) =    \mathds{1}_{ \{ \tau= t_i \} }  \mathds{1}_{(0,t_i]} (t).
\end{equation}
Hence, by \eqref{decomp_stochastic_int12}, \eqref{local001}, \eqref{decomp_stochastic_int102}, and \eqref{localprop00}, we have:
\begin{equation}
    \label{local0}
    \begin{split}
    & \mathds{1}_{[\![  0, \tau]\!]} (t)  \int_0^t \int_{\mathbb{R}^d} \mathds{1}_{A}(s,y) \mathds{1}_{[\![  0, \tau]\!]} (s) \tilde{\Lambda}(ds,dy) \\
    & =  \sum_{n = 1 }^{N_0 +1}  \mathds{1}_{ \{ \tau= t_n \} }  \mathds{1}_{[\![  0, \tau]\!]} (t) \int_{0}^{t} \int_{\mathbb{R}^d } \mathds{1}_{A}(s,y) \mathds{1}_{(0,t_n]}  (s) \tilde{\Lambda}(ds,dy) \\
    &=  \sum_{n = 1 }^{N_0 +1}  \mathds{1}_{ \{ \tau= t_n \} }  \mathds{1}_{(0,t_n]} (t) \int_{0}^{t} \int_{\mathbb{R}^d } \mathds{1}_{A}(s,y) \mathds{1}_{(0,t_n]} (s) \tilde{\Lambda}(ds,dy) \\ 
    &=  \sum_{n = 1 }^{N_0 +1 }  \mathds{1}_{ \{ \tau= t_n \} }  \mathds{1}_{(0,t_n]} (t) \int_{0}^{t} \int_{\mathbb{R}^d } \mathds{1}_{A}(s,y)  \tilde{\Lambda}(ds,dy) \\ 
    &=  \mathds{1}_{ {[\![  0, \tau]\!]} } (t)  \int_{\mathbb{R}^d } \mathds{1}_{A}(s,y)  \tilde{\Lambda}(ds,dy). \\ 
    \end{split}
\end{equation}

\textit{Step 2:} Now we prove \eqref{local-prop0} when \(\tau\) is an arbitrary stopping time. Note that there exists a sequence of elementary stopping times \(\{ \tau_n \}_{n \in \mathbb{N}}\) such that \(\tau_n \uparrow \tau\) a.s. as \(n \to +\infty\). Hence, by Theorem \ref{DCT1}, we have that \eqref{local-prop0} holds for \(\tau\).

\textit{Step 3:} We prove the case when \( H(t,x) \) is predictable and \( H(t,x) \mathds{1}_{[\![ 0, T_N]\!]}(t) \) is \( p \)-integrable with respect to \( \tilde{\Lambda} \) for all \( n \in \mathbb{N} \). Since \( H(s,y) \mathds{1}_{ {[\![ 0, T_N]\!]} }(s) \in L^{1,p}(\tilde{\Lambda}) \) for all \( N \in \mathbb{N} \), there exists a sequence of simple integrands \( \{ S_n \}_{n \in \mathbb{N}} \) such that \( \| v^{(t)} - S_n^{(t)} \|_{\tilde{\Lambda},p} \to 0 \) as \( n \to +\infty \), where \( v^{(t)}(s,y) = H(s,y) \mathds{1}_{[0,t]}(s) \) and \( S_n^{(t)}(s,y) = S_n(s,y) \mathds{1}_{[0,t]}(s) \) for each fixed \( t \in \mathbb{R}_+ \). Hence, by the linearity of \( I^{\tilde{\Lambda}} \) and the steps above, we have:
\begin{equation}
\label{ftime}
\mathds{1}_{[\![ 0, \tau]\!]}(t) \int_0^t \int_{\mathbb{R}^d} S_n(s,y) \tilde{\Lambda}(ds,dy) = \mathds{1}_{[\![ 0, \tau]\!]}(t) \int_0^t \int_{\mathbb{R}^d} S_n(s,y) \mathds{1}_{[\![ 0, \tau]\!]}(s) \tilde{\Lambda}(ds,dy),
\end{equation}
for all \( n \in \mathbb{N} \). Letting \( n \to +\infty \) in \eqref{ftime}, we conclude that \eqref{local-prop-time} holds.

\end{proof}

\section*{Acknowledgments}
The author is deeply grateful to Raluca M. Balan for providing valuable comments and insightful discussions. The author is also thankful for the constructive feedback provided by the anonymous referee, all of which led to improvements in Section 3.


\begin{thebibliography}{9}

\bibitem{balan27}  R. M. Balan. SPDEs with $\alpha$-stable L\'evy noise: a random field approach. \textit{Int. J. Stoch. Anal.}, Article ID 793275, 22 pages, 2014.

\bibitem{balan28} R. M. Balan. Integration with respect to L\'evy colored noise, with applications to SPDEs. \textit{Stochastics}, 87(3):363–381, 2015.

\bibitem{balan33} R. M. Balan and C.B. Ndongo. Intermittency for the wave equation with L\'evy white noise. \textit{Statist. Probab. Lett.}, 109:214–223, 2016.

\bibitem{balan38} R. M. Balan and C.B. Ndongo. Malliavin differentiability of solutions of SPDEs with L\'evy white noise. \textit{Int. J. Stoch. Anal.}, Article ID 9693153, 9 pages, 2017.

\bibitem{balan49} R. M. Balan. Stochastic wave equation with L\'evy white noise. \textit{ALEA, Lat. Am. J. Probab. Math. Stat.}, 20:463–496, 2023.

\bibitem{berger0} Q. Berger and H. Lacoin. The scaling limit of the directed polymer with power-law tail disorder. \textit{Comm. Math. Phys.}, 386(2):1051–1105, 2021.

\bibitem{berger} Q. Berger, C. Chong and H. Lacoin. The stochastic heat equation with multiplicative L\'evy noise: Existence, moments, and intermittency. \textit{Commun. Math. Phys.}, 402(3):2215-2299, 2023.

\bibitem{bit1} K. Bichteler and J. Jacod. Random measures and stochastic integration. In G. Kallianpur (Ed.), \textit{Theory and Application of Random Fields}, pages 1–18. Springer, Berlin, 1983.

\bibitem{bit2} K. Bichteler. \textit{Stochastic Integration with Jumps}. Cambridge University Press, Cambridge, 2002.

\bibitem{chong0} C. Chong and C. Klüppelberg. Integrability conditions for space–time stochastic integrals: Theory and applications. \textit{Bernoulli}, 21(4):2190-2216, 2015.

\bibitem{chong1} C. Chong. Stochastic PDEs with heavy-tailed noise. \textit{Stochastic Process. Appl.}, 127(7):2262-2280, 2017.

\bibitem{chong2} C. Chong. L\'evy-driven Volterra equations in space and time. \textit{J. Theoret. Probab.}, 30(3):1014–1058, 2017.

\bibitem{CDH} C. Chong, R. C. Dalang and T. Humeau. Path properties of the solution to the stochastic heat equation with L\'evy noise. \textit{Stoch. Partial Differ. Equ. Anal. Comput.}, 7(1):123–168, 2019.

\bibitem{Dalang99} R. C. Dalang. Extending the martingale measure stochastic integral with applications to spatially homogeneous SPDE's. \textit{Electron. J. Probab.}, 4(6):1-29, 1999.

\bibitem{DalangWave} R. C. Dalang. The stochastic wave equation. In R. C. Dalang, D. Khoshnevisan, D. Nualart, and Y. Xiao (Eds.), \textit{A Minicourse on Stochastic Partial Differential Equations}, pages 39–71. Springer, Berlin, 2009.

\bibitem{Dalang-Quer} R. C. Dalang and L. Quer-Sardanyons. Stochastic integrals for SPDE’s: A comparison. \textit{Expo. Math.}, 29(1):67-109, 2011.

\bibitem{HM} T. Humeau. \textit{Stochastic partial differential equations driven by L\'evy white noises}. Thesis (Ph.D.), No. 8223, EPFL, Lausanne, 2017.

\bibitem{JS} J. Jacod and A. N. Shiryaev. \textit{Limit Theorems for Stochastic Processes}, second ed., Springer, Berlin, 2003.

\bibitem{Millet-Sole99} A. Millet and M. Sanz-Solé. A stochastic wave equation in two space dimensions: smoothness of the law. \textit{Ann. Probab.}, 27(2):803-844, 1999.

\bibitem{RR} B. S. Rajput and J. Rosinski. Spectral representations of infinitely divisible processes. \textit{Probab. Theory Relat. Fields}, 82(3):451–487, 1989.

\bibitem{SLB} E. Saint Loubert Bié. Étude d'une EDPS conduite par un bruit Poissonnien. \textit{Probab. Theory Related Fields}, 111(2):287-321, 1998.

\bibitem{ves} F. Tr\'eves. \textit{Basic Linear Partial Differential Equations}. Vol. 62. Academic Press, 1975.

\bibitem{walsh86} J. B. Walsh. An introduction to stochastic partial differential equations. In P. L. Hennequin (Ed.), \textit{École d'Été de Probabilités de Saint-Flour XIV-1984}, pages 265–439. Springer, Berlin, 1986.

\end{thebibliography}
\end{document}